\documentclass[11pt]{amsart}

\usepackage{color}

\usepackage{tikz-cd}
\usepackage[all]{xy}
\usepackage[utf8x]{inputenc}
\usepackage[english]{babel}
\usepackage{mathtools}
\usepackage[T1]{fontenc}
\usepackage{lmodern}
\usepackage{mathrsfs}
\usepackage{enumitem}

\usepackage{graphicx,epsfig,amscd,graphics,xypic, subcaption}
\usepackage[all]{xy}
\usepackage{amsmath,amssymb}

\usepackage[margin=4cm]{geometry}

\usepackage{hyperref}
\hypersetup{
    bookmarks=true,         
    unicode=false,          
    pdftoolbar=true,        
    pdfmenubar=true,        
    pdffitwindow=false,     
    pdfstartview={FitH},    
    pdftitle={Cascades in the dynamics of affine interval exchange transformations},    
    pdfauthor={Adrien Boulanger, Charles Fougeron, Selim Ghazouani},     
    pdfsubject={Dynamics of the Disco surface},   
    pdfcreator={LaTeX},   
    pdfproducer={}, 
    pdfkeywords={Dynamical systems} {Dilatation surfaces} {Rauzy-Veech induction}, 
    pdfnewwindow=true,      
    colorlinks=false,       
    linkcolor=red,          
    citecolor=green,        
    filecolor=magenta,      
    urlcolor=cyan           
}

\newcommand{\V}{\mathbf{V}_{\Sigma}}

\newcommand{\C}{\mathbb{C}}
\newcommand{\Z}{\mathbb{Z}}
\newcommand{\R}{\mathbb{R}}

\newcommand{\Hy}{\mathbb{H}}

\newcommand{\Diff}{\mathrm{Diff}}

\newcommand{\SL}{\mathrm{SL}(2,\R)}
\newcommand{\MCG}{\mathrm{MCG}(\Sigma)}

\newcommand{\norm}[1]{\left| #1 \right|}

\newcommand{\resp}{\textit{(resp. }}
\newcommand{\disco}{\Sigma}
\newcommand{\discoveech}{\mathbf{V}_{\disco}}
\newcommand{\veechsg}{\Gamma}

\newcommand{\Sone}{S^1}

\newtheorem{theorem}[equation]{Theorem}
\newtheorem*{theorem*}{Theorem}
\newtheorem{corollary}[equation]{Corollary}
\newtheorem{lemma}[equation]{Lemma}
\newtheorem*{property*}{Property}
\newtheorem{proposition}[equation]{Proposition}

\newtheorem{remark}[equation]{Remark}
\theoremstyle{definition}
\newtheorem{definition}[equation]{Definition}

\newcommand{\fonctionbis}[4]{\begin{array}{r c l}
					   #1 & \to & #2 \\
					 #3 & \mapsto & #4 \\
			   \end{array}}
\newcommand{\quotient}[2]{{\raisebox{.2em}{$#1$}\left/\raisebox{-.2em}{$#2$}\right.}}

\newcommand{\T}{\mathcal{T}(\alpha)}

\title[Cascades in the dynamics of AIETs]{Cascades in the dynamics of affine interval exchange transformations}

\author{Adrien Boulanger}
\author{Charles Fougeron}
\author{Selim Ghazouani}

\begin{document}

\maketitle

\begin{abstract}

  We describe in this article the dynamics of a $1$-parameter family
  of affine interval exchange transformations. This amounts to studying
  the directional foliations of a particular dilatation surface introduced
  in \cite{DFG}, the \textit{Disco surface}.  We show that this family
  displays various dynamical behaviours: it is generically
  \textit{dynamically trivial} but for a Cantor set of parameters the
  leaves of the foliations accumulate to a (transversely) Cantor
  set. This study is achieved through the analysis of the dynamics of the Veech group of this surface combined a modified version of Rauzy induction in the context of affine interval exchange transformations.

\end{abstract}

\section{Introduction.}

An \textit{affine interval exchange transformation} (or AIET) is a piecewise continuous bijection of the interval $[0,1]$ which is affine restricted to its intervals of continuity. It has been known since the work of Levitt (\cite{Levitt}) that AIETs can display as complicated a topological behaviour as dimension one allows: it can either be asymptotically periodic, minimal or (and this is  the surprising part) have an invariant quasi-minimal Cantor set. In the latter case, the AIET would still be semi-conjugated to a minimal \textit{linear} interval exchange transformation. In the spirit of generalising the theory of circle diffeomorphisms to piecewise continuous bijection of the interval, Camelier-Guttierez (\cite{CamelierGutierrez}) begun a study of the regularity of the conjugacy between affine and linear IET, pursued by Cobo (\cite{Cobo}), Bressaud-Hubert-Maas(\cite{BressaudHubertMaass}) and concluded by Marmi-Moussa-Yoccoz(\cite{MarmiMoussaYoccoz}) who proved that almost every linear IET can be semi-conjugated to an AIET with an invariant Cantor set, in sharp contrast with Denjoy theorem in the case of sufficiently regular diffeomorphisms of the circle.

\vspace{2mm} \noindent The goal of this article is to initiate a systematic study of the generic dynamical behaviour in parameter families of AIETs. The standard result in the theory of circle diffeomorphisms is a theorem by Herman (see \cite{Herman}) predicting that for any (sufficiently regular) one-parameter family of circle diffeomorphisms, the set of minimal parameters has non-zero Lebesgue measure. On the other hand, it was known since the seminal work of Peixoto (see \cite{Peixoto}, \cite{Peixoto1}) that asymptotically periodic behaviour is \textit{topologically generic} for flows on closed surfaces \footnote{Generalised interval exchange transformations, of which AIETs are particular cases, should be thought of as first return maps of flows on higher genus surfaces}, and a refinement of this theorem was proved by Liousse (\cite{Liousse}) for transversally affine foliations in the case of higher genus surfaces. We present in this article a one parameter family of AIETs whose generic behaviour (in the measure theoretic sense) contrasts with the case of circle diffeomorphisms and Herman's theorem.

 \vspace{4mm}

\noindent We consider the map $F : D \longrightarrow D$, where
$D = [0,1[$, defined the following way: \vspace{-.2cm}
$$
\def\arraystretch{1.8}
\begin{array}{ccrl}
  \text{if} \ x \in \left[ 0,\frac{1}{6}\right[             & \text{then} & F(x) = & 2x + \frac{1}{6} \\
  \text{if} \ x \in \left[ \frac{1}{6},\frac{1}{2}\right[   & \text{then} & F(x) = & \frac{1}{2}(x-\frac{1}{6}) \\
  \text{if} \ x \in  \left[\frac{1}{2},\frac{5}{6}\right[   & \text{then} & F(x) = & \frac{1}{2}(x-\frac{1}{2}) + \frac{5}{6} \\
  \text{if} \ x \in  \left[\frac{5}{6},1\right[             & \text{then} & F(x) = & 2(x-\frac 5 6) + \frac{1}{2}
\end{array}
$$

\begin{figure}[h]
  \centering
  \includegraphics[width=.36\linewidth]{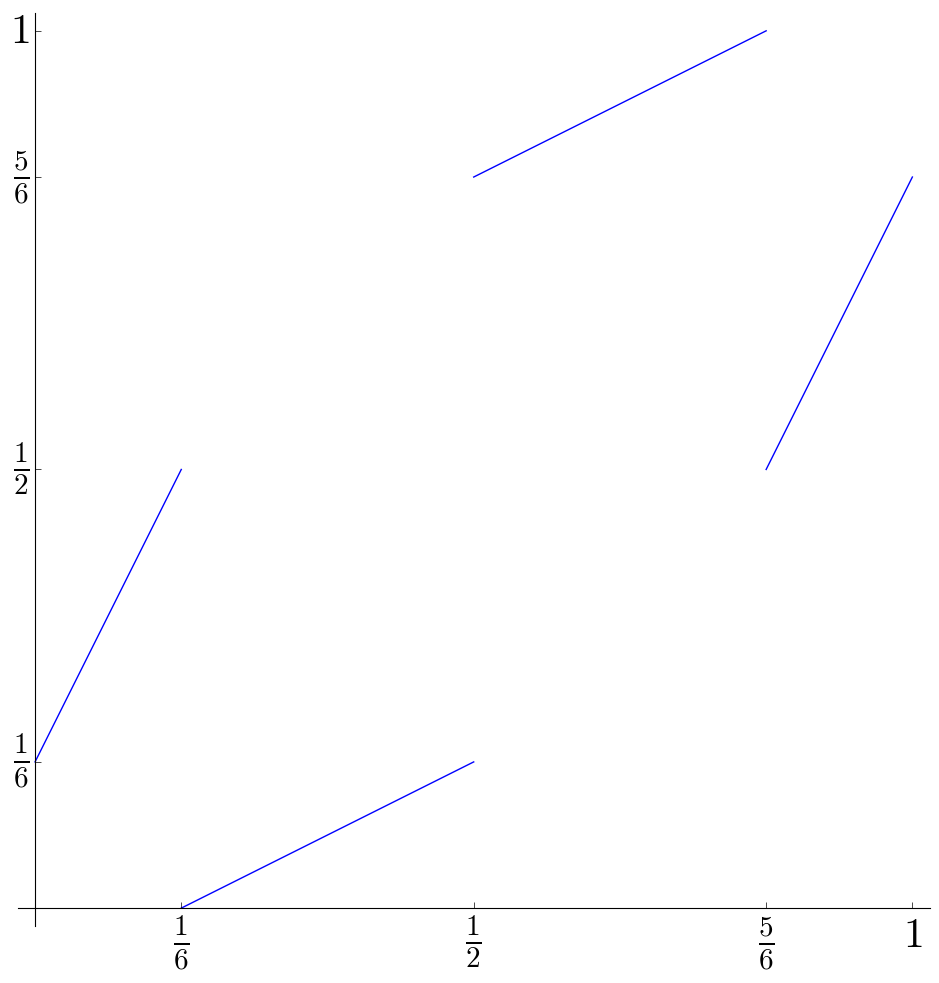}
  \caption{The graph of $F$.}
\end{figure}

\vspace{2mm}

\noindent The map $F$ is an \textit{affine interval exchange
  transformation} (AIET) and one easily verifies that for all
$x \in D$, $F^2(x) = x$. Its dynamical behaviour is therefore as
simple as can be. Composing given maps by a family of linear rotations is a simple way to produce families of maps of the interval. Thus we consider the family $(F_t)_{t \in \Sone}$,
parametrised by $\Sone = \quotient{\R}{\Z}$ defined by
$$ F_t = F \circ r_t $$ where $r_t : [0,1[ \longrightarrow [0,1[$
is the translation by $t$ modulo $1$.  \\

\noindent  The following definition is of crucial importance for what follows. It was introduced by Liousse in \cite{Liousse} who proved that this dynamical behaviour is topologically generic for transversally affine foliations on surfaces.

\noindent Recall that the orbit of a point $x$ under a map $f$ is the set $\mathcal{O}(x) = \{ f^n(x) \ | \ n \in \mathbb{N} \}$ and its $\omega$-limit is the set of accumulation points of the sequence $(f^n(x))_{n \in \mathbb{N}}$.

\begin{definition} We say that $F_t$ is dynamically trivial if there exists two periodic
points $x^+, x^- \in D$ of orders $p,q \in \mathbb{N}$ such that

\begin{itemize}

\item $(F_t^p)'(x^+) < 1$
\item $(F_t^q)'(x^-) > 1$
\item for all $z \in D$ which is not in the orbit of $x^-$, the
  $\omega$-limit of $z$ is equal to $\mathcal{O}(x^+) $ the orbit of
  $x^+$.

\end{itemize}
\end{definition}

\noindent It means that the map $F_t$ has two periodic orbits, one of which attracting all the other orbits but the other periodic orbit which is repulsive. The following picture is the product of a numerical experiment representing periodic orbits in the family $(F_t)$ and their bifurcations. \vspace{2mm}

\begin{figure}[!h]
  \centering
  \includegraphics[scale=0.4]{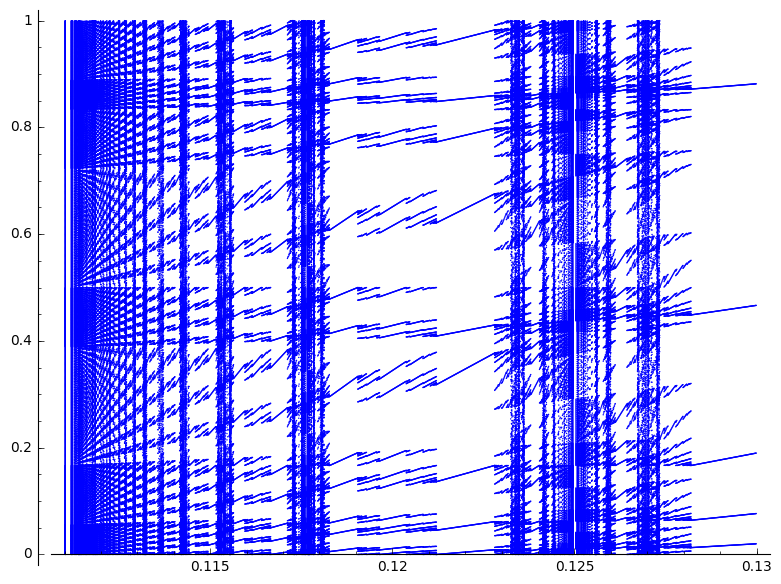}
  \caption{The $\omega$-limit of a random point for $F_t$,
    for $ 0.11 \leq t \leq 0.13$. Parameters with periodic orbits are open and dense, and can accumulate to seemingly minimal parameters.}
  \label{disco2}
\end{figure}

\noindent This article aims at highlighting that this one-parameter family of
AIETs displays rich and various dynamical behaviours. The analysis
developed in it, using tools borrowed from the theory of geometric
structures on surfaces, leads to the following theorems.

\begin{theorem}
  \label{thm1}
  For Lebesgue-almost all $t \in \Sone$, $F_t$ is dynamically trivial.
\end{theorem}

Our theorem is somewhat a strengthening of
Liousse's theorem for this $1$-parameter family of AIETs and a counterexample to Herman's in higher genus. Indeed, we
prove that this genericity is also of measure theoretical nature. It is also worth pointing out that a lot of parameters in this family correspond to attracting
\textit{exceptional minimal sets} (\textit{i.e.} which are homeomorphic to a Cantor set).

\begin{theorem}
  \label{thm2}
  For all $t$ in a Cantor set of parameters in $\Sone$ there
  exists a Cantor set $\mathcal{C}_t \subset D$ such that for all
  $x \in D$, the $F_t$ $\omega$-limit of $x$ is equal to
  $\mathcal{C}_t $.
\end{theorem}

\noindent The remaining parameters form a Cantor set denoted by
$\Lambda_{\Gamma} \subset \Sone$. This notation is borrowed from
Fuchsian group theory as we will indeed see that this Cantor set
is the limit set of a subgroup $\Gamma < \mathrm{PSL}(2,\R)$.
For parameters in $\Lambda_{\Gamma} $, we have,

\begin{theorem}
\label{thm3}
Let $\mathcal{H}$ be the set of points in $\Lambda_{\Gamma}$ which are not fixed by a parabolic element of $\Gamma$. Then

\begin{itemize}
  \item for $\theta \in \mathcal{H}$, the foliation is not dynamically trivial;
  \item for $\theta \in \Lambda_{\Gamma} \setminus \mathcal{H}$  the foliation is totally periodic.
\end{itemize}
\end{theorem}

The foliations corresponding to directions in $\mathcal{H}$ are also not totally periodic.
Extensive computer experiments give evidences that these foliations are minimal.

\subsection*{Outline of the paper.} Sections \ref{secaffsur} and \ref{veechgroup} are devoted to recalling geometric basics about dilatation surfaces and to the study of the hidden symmetries of the family $(F_t)$ using this geometric perspective. Section \ref{secgendir} is mostly independent of the rest of the article. Therein we explain how to generalise the renormalisation procedure known as Rauzy-Veech induction to the context of piecewise contracting maps of the interval. This analysis allows the understanding of the dynamical behaviour of $F_t$ for sufficiently many parameters so that we can rely on the aforementioned symmetries to reach almost every parameter, which we explain in Section \ref{secglopic}.

	\subsubsection*{The discosurface} The first step of the proof consists in associating to the family $(F_t)_{t \in \Sone}$ a dilatation surface which we denote by $\disco$  obtained by the gluing represented in Picture \ref{disco2}.

 \begin{figure}[!h]
   \centering
   \includegraphics[scale=0.7]{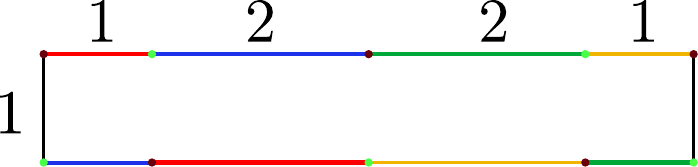}
   \caption{The surface $\disco$.}
   \label{disco2}
 \end{figure}

\noindent As a dilatation surface $\disco$ is naturally endowed with a family of foliations which we call directional foliations. For definitions of both dilatation surfaces and these foliations we refer to Section \ref{subsec dilatation }. Our family of AIETs $(F_t)_{t\in[0,1[}$ and these foliations are linked by the fact that the directional foliation in
 direction $\theta$ admits $F_t$ as their first return map on a
 cross-section, for $t = \frac{6}{\tan \theta}$. In particular they
 share the same dynamical properties hence the study of the family
 $F_t$ reduces to the study of the directional foliations of $\disco$.

 \subsubsection*{The Veech group of $\disco$.}

 The major outcome of this change of point of view is the appearance of
 hidden symmetries. Indeed, the surface $\disco$ has non-trivial group of
 affine symmetries, \textit{i.e.} a non-trivial group of diffeomorphisms given in charts as an element of the affine group $\mathrm{GL}(2, \R) \rtimes \mathbb{R}^2$ of $\mathbb{R}^2$. All this material is defined in Section \ref{subsecveech}. Such a group of affine diffeomorphisms admits a natural representation in $ \mathrm{SL}(2, \R)$; we call the image of this representation the Veech group that we denote by $\V$. This new group naturally acts on the set of directions of $\mathbb{R}^2$. The directions $\theta$ which are $\SL$-equivalent through $\discoveech$ correspond to two foliations which are conjugated thus sharing the same dynamical behaviour. This remark will allow us to considerably reduce the number of parameters $\theta$ (equivalently $t$) that we need to analyse.
 \vspace{2mm}

 \noindent  Using a standard construction of affine diffeomorphisms using flat cylinder decompositions recalled in Subsection \ref{subsecgroup}, we show that the group $\discoveech$  is discrete and contains the following group

$$ \veechsg = \big\langle \begin{pmatrix}
  1 & 6\\
  0 & 1
\end{pmatrix}, \begin{pmatrix}
  1 & 0\\
  \frac{3}{2} & 1
\end{pmatrix}, \begin{pmatrix}
  -1 & 0\\
  0 & -1
\end{pmatrix} \big\rangle. $$

\noindent The matrix $- \mathrm{Id}$ belonging to $\V$ it is natural to project $\veechsg$ to $\mathrm{PSL}(2, \R)$. We will hereafter make the slight abuse of notation to denote the image of $\veechsg$ by this projection $\veechsg$ as well. This group is a Schottky group of rank $2$. The study of this action is performed in Section \ref{veechgroup} and leads to the following. \\

\begin{itemize}

\item There is a Cantor set $\Lambda_{\veechsg} \subset \mathbb{RP}^1$
  of  measure  zero on which $\veechsg$ acts minimally
  ($\Lambda_{\veechsg}$ is the limit set of $\Gamma$).

\item The action of $\veechsg$ on
  $\Omega_{\veechsg} = \mathbb{RP}^1 \setminus \Lambda_{\veechsg}$ is
  properly discontinuous and the quotient is homeomorphic to a circle
  ($\Omega_{\veechsg}$ is the discontinuity set of $\Gamma$). It
  allows us to identify a "small" fundamental domain
  $I \subset \mathbb{RP}^1$ such that the description of the
  dynamics of foliations in directions $\theta \in I$ implies the
  description for every parameter in $\Omega_{\veechsg}$ (which is an open set of full measure). \\
\end{itemize}

\noindent Note that the Cantor set $\Lambda_{\Gamma}$ has nothing to do with the one described in Theorem \ref{thm2}. The latter is a subset of $\Omega_{\Gamma}$.

\subsubsection*{Affine Rauzy-Veech induction.} The study of the directional foliations for $\theta \in I$ reduces to the understanding of the dynamics of piecewise contracting affine $2$-intervals maps. To perform the dynamical study of these applications we adapt in this $2$-contracting intervals setting a well known re-normalisation procedure, the Rauzy-Veech induction. The outcome of this method may be summarised as follows:

\begin{itemize}
\item there is a Cantor set of measure zero of parameters
  $\theta \in \Omega_{\Gamma}$ for which the associated foliation
  accumulates to a set which is locally a product of a Cantor set with an interval ;
\item other directions in $\Omega_{\Gamma}$ are dynamically
  trivial.
\end{itemize}

\vspace{2mm}

\noindent A remarkable corollary of the understanding of the dynamics
of the directions in $I$ is the complete description of $\V$:

\begin{theorem}
  \noindent The Veech group of $\Sigma$ is exactly $\Gamma$.
\end{theorem}

\noindent The proof is a rather straightforward corollary of the dynamical description.
We prove that the limit set of $\V$ is actually the same as the one of $\Gamma$, and
conclude using some elementary geometric arguments to prove that these groups are equal. \\

\paragraph*{\bf Acknowledgements.}
We are grateful to Bertrand Deroin for his encouragements and the
interest he has shown in our work. We are also very thankful to Pascal
Hubert for having kindly answered the very many questions we asked
him, to Vincent Delecroix, to Nicolas Tholozan for interesting discussions around the proof
of Theorem \ref{veechgroup1} and to Matt Bainbridge.

\section{Dilatation surfaces and their Veech groups.}

\label{secaffsur}

We introduce in this section geometric objects which
will play a role in this paper. This includes the definition of a dilatation surface,
their associated foliations as well as their Veech groups, the construction of the
Disco surface and how it is linked to our family of AIETs. We also compute
explicitly two elements of the Veech group of the Disco surface.

\subsection{Dilatation surfaces and their foliations.}
\label{subsec dilatation }

\begin{definition}
	\label{defdilatation}
  A \textit{dilatation surface} is a surface $\Sigma$ together with a
  finite set $S \subset \Sigma$ and an atlas
  $\mathcal{A} = (U_i, \varphi_i)_{i \in I}$ on $\Sigma \setminus S$
  whose charts $\varphi_i$ take values in $\C$ such that

  \begin{itemize}

  \item the transition maps are locally restriction of elements of
    $\mathrm{Aff}_{\R^*_+}(\C) = \{ z \mapsto az+b \ | \ a \in \R, \ a>0,  \ b \in \C \} $;

  \item each point of $S$ has a punctured neighbourhood which is
    affinely equivalent to the $k$-sheets covering of $\mathbb{C}^*$.

  \end{itemize}

  \noindent We call an element of the set $S$ a \textit{singularity}
  of the dilatation surface $\Sigma$.

\end{definition}

\noindent This definition is rather formal, and the picture one has to have in mind is that a dilatation surface is what one gets when you take a union of Euclidean polygons and glue together pairs of oriented parallel sides along the unique complex affine transformation that sends one to the other.  \\

\subsubsection{The Disco surface} The surface we are about to define will be the main object of interest
of this text. It is the surface obtained after proceeding to the
gluing below:

\begin{figure}[!h]
  \centering
  \includegraphics[scale=0.8]{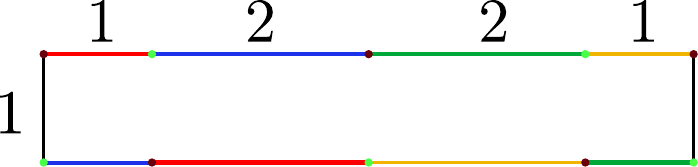}
  \caption{The surface $\disco$. Each side of the same color is identified with the corresponding one by the unique complex affine map of the form $z \mapsto a z + b $ with $a >0$. }
  \label{disco}
\end{figure}

\noindent We call the resulting surface 'Disco' surface. In the following
$\disco$ will denote this particular surface. This is a genus $2$ dilatation surface which has two singular
  points of angle $4\pi$. They correspond to the vertices of the
  polygon drawn in Figure \ref{disco}. Green (light) ones project onto one
  singular point and brown (dark) ones project onto the other.

\subsubsection{Foliations and saddle connections. }
\label{subsec foliation}
Together with a dilatation surface comes a natural family of
foliations. Fix an angle $\theta \in S^1$ and consider the trivial
foliation of $\C$ by straight lines directed by $\theta$. This
foliation being invariant by the action of
$\mathrm{Aff}_{\R^*_+}(\C)$, it is well defined on
$\Sigma \setminus S$ and extends at points of $S$ to a singular
foliation on $\Sigma$ such that its singular type at a point of $S$ is
saddle-like. We denote this family of foliations by $(\mathcal{F}_{\theta})_{\theta \in S^1}$.
\vspace{3mm}

\noindent A \textit{saddle connection} on $\Sigma$ is a singular leave that goes from a singular point to another. The set of saddle connections of a dilatation surface is
countable hence so is the set of directions having saddle connections.\\

\noindent In the case of the Disco surface, one can easily draw these foliations on its polygonal model : they correspond to the restriction of the directional foliations of $\mathbb{R}^2$ to the polygon.  One can check that the horizontal curve on the picture below is actually a cross-section for every foliation $\mathcal{F}_{\theta}$ with $\theta \neq 0$.

\begin{figure}[!h]
  \centering
  \includegraphics[scale=0.8]{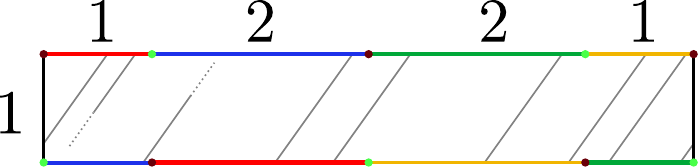}
  \caption{The surface $\disco$ and a leaf of a directional foliation.}
  \label{disco}
\end{figure}

\noindent The first return map $\varphi_{\theta}$ of the foliation $\mathcal{F}_{\theta}$ with respect to this cross-section satisfies :

$$ \varphi_{\theta} =  F_t \hspace{0.5cm} \text{with} \hspace{0.5cm} t = \frac{6}{\tan\theta} \ .$$

\subsection{The Veech group of a dilatation surface.}
\label{subsecveech}

Let $\Sigma$ be a dilatation surface and $g \in \Diff^+(\Sigma)$ an \textit{affine
  diffeomorphism} of $\Sigma$, namely a diffeomorphism which reads in dilatation coordinates as an element of the affine group $\mathrm{GL}^+(2,\R) \ltimes \R^2$ of $\mathbb{R}^2$ with the
standard identification $\C \simeq \R^2$ (more explicitly, a map of the form
$$ \begin{pmatrix}
  x \\ y
\end{pmatrix} \mapsto A \begin{pmatrix} x \\ y
\end{pmatrix} + B $$
where $A \in \mathrm{GL}_+(2,\mathbb{R})$ and $B$ is a vector
of $\mathbb{R}^2$). We denote by $\mathrm{Affine}(\Sigma)$ the subgroup of $\Diff^+(\Sigma)$ of
affine diffeomorphisms. The linear part in coordinates of an element of
$\mathrm{Affine}(\Sigma)$ is well defined up to multiplication by a
constant $\lambda \in \R^*_+$. This gives rise to a well-defined
morphism:
$$ \rho : \mathrm{Affine}(\Sigma) \longrightarrow \mathrm{SL}(2, \R) $$

\noindent which to an affine diffeomorphism associates its \textit{normalised}
linear part. We call this morphism the \textit{Fuchsian representation}.

\begin{remark}
  It is important to understand that the fact that the image of $\rho$
  lies in $\SL$ is somewhat artificial and that the space it naturally
  lies in is $\mathrm{GL}^+(2,\R)/ \R^*_+$. In particular, when an
  element of the Veech group is looked at in charts, there
  is no reason the determinant of its derivative should be equal to
  $1$, however natural the charts are.
\end{remark}

\begin{definition} The image of the Fuchsian representation $\rho(\mathrm{Affine}(\Sigma))$ is called the \textit{Veech group} of $\Sigma$ and is denoted by $\V$. The Veech group naturally acts on the circle $\mathbb{S}^1$, we will refer hereafter to this action as the \textbf{projective action} of the Veech group.
\end{definition}

\noindent The key point is that such an affine diffeomorphism $g$ maps the $\theta$-directional foliation
onto the foliation associated to the direction $ \rho(g) (\theta) $, in particular these two foliations are conjugated and therefore have the same dynamical behaviour. This allows us to reduce the amount of directional foliations to study to the set of parameters corresponding to the quotient of the circle $S^1$ by the projective action of the Veech group.

\subsubsection{About the Veech group of $\disco$.} This subsection is devoted to computing two elements of the Veech group. We utilise a method which is standard for translation surfaces, which consists in decomposing $\Sigma$ into flat cylinders of commensurable moduli and to let the multi-twist associated act affinely on each cylinder as a parabolic element.  \\
\label{subsecgroup}

\paragraph{\bf Flat cylinders}

A \textit{flat cylinder} is the dilatation surface you get when gluing two
opposite sides of a rectangle. The \textit{height} of the cylinder is
the length of the sides glued together and its \textit{width}
is the length of the non-glued sides, that is the boundary components
of the resulting cylinder. Of course, only the ratio of these two
quantities is actually a well-defined invariant of the flat cylinder,
seen as a dilatation surface. More precisely, we define

$$ m = \frac{\text{width}}{\text{height}}$$ and call this quantity
the \textit{modulus} of the associated flat cylinder.

\noindent If $C$ is a cylinder of modulus $m$, there is an element of
$f \in \mathrm{Affine}(C)$ which has the following properties:

\begin{itemize}

\item $f$ is the identity on $\partial C$;

\item $f$ acts as a unique Dehn twist of $C$;

\item the matrix associated to $f$ is $\begin{pmatrix}
    1 & m \\
    0 & 1
  \end{pmatrix}$, if $\partial C$ is assumed to be in the horizontal
  direction.

\end{itemize}

\paragraph{\bf Decomposition in flat cylinders and parabolic elements
  of the Veech group.}

We say a dilatation surface $\Sigma$ has a \textit{decomposition in flat
  cylinders} in a given direction(say the horizontal one) if there
exists a finite number of saddle connections in this direction whose
complement in $\Sigma$ is a union of flat cylinders. If additionally
the flat cylinders have commensurable moduli, the Veech group of
$\Sigma$ contains the matrix $$\begin{pmatrix}
  1 & m'\\
  0 & 1
\end{pmatrix}$$ where $m'$ is the smallest common multiple of all the
moduli of the cylinders appearing in the cylinder decomposition. If
the decomposition is in an another direction $\theta$, the Veech group
actually contains the conjugate of this matrix by a rotation of angle
$\theta$. Moreover, an affine diffeomorphism realising this matrix
is a Dehn twist along the multi-curve made of all the simple closed
curves associated to each of the cylinders of the decomposition.

\vspace{2mm}

\paragraph{\bf Calculation of elements of the Veech group of
  $\disco$.}  The above paragraph allows us to bring to light two
parabolic elements in $\discoveech$. Indeed, $\disco$ has two cylinder
decompositions in the horizontal and vertical direction.

\begin{itemize}

\item The decomposition in the horizontal direction has one cylinder
  of modulus $6$, represented Figure \ref{cylinder1prime} below.

\begin{figure}[!h]
  \centering
  \includegraphics[scale=0.8]{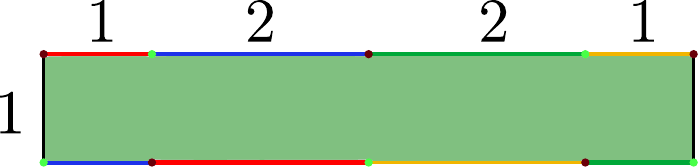}
  \caption{Cylinder decomposition in the horizontal direction.}
  \label{cylinder1prime}
\end{figure}

\noindent Applying the discussion of the last paragraph gives, we get
that the matrix
$\begin{pmatrix}
  1 & 6\\
  0 & 1
\end{pmatrix}$
belongs to $\discoveech$.\\

\item The decomposition in the vertical direction has two cylinders,
  both of modulus $\frac{3}{2}$, represented in Figure \ref{cylinder2} below.

\begin{figure}[!h]
  \centering
  \includegraphics[scale=0.8]{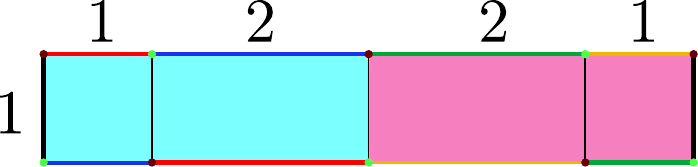}
  \caption{Cylinder decomposition in the vertical direction.}
  \label{cylinder2}
\end{figure}

\noindent Again, we get that the matrix
$\begin{pmatrix}
  1 & 0\\
  \frac{3}{2} & 1
\end{pmatrix}$
belongs to $\discoveech$.

\end{itemize}

\noindent Finally notice that both the polygon and the gluing
pattern we used to build $\disco$ is invariant by the rotation of angle $\pi$,
which implies that the matrix
$$\begin{pmatrix}
  -1 &  0\\
   0 & -1
\end{pmatrix}$$
is realised by an involution in
$\mathrm{Affine}(\disco)$. Putting all the pieces together we get,

\begin{proposition}
  The group

  $$
  \left< A =
  \begin{pmatrix}
    1 & 6 \\
    0 & 1
  \end{pmatrix},
  \hspace{0.3cm}
  B =
  \begin{pmatrix}
              1 & 0 \\
    \frac{3}{2} & 1
  \end{pmatrix},
  \hspace{0.3cm}
  - \mathrm{Id} \right> $$
  is a subgroup of $\discoveech$.
\end{proposition}

\section{The hyperbolic geometry of $\Gamma$ }
\label{veechgroup}

\subsection{ The subgroup $\veechsg$ }
\label{subsecgroup2}
We computed in Section \ref{secaffsur} three elements $A$, $B$ and $-\mathrm{Id}$ of the Veech group
of $\disco$. The presence of the matrix $-\mathrm{Id}$ in $\V$ indicates
that directional foliations on the surface $\disco$ are invariant by
reversing orientation. This motivates the study of the Veech group action on
$\mathbb{RP}^1 := \quotient{S^1}{\mathrm{-Id}}$ instead of
$S^1$. \newline

\noindent We will often identify $\mathbb{RP}^1$ to the interval
$[-\frac{\pi}{2},\frac{\pi}{2})$ by using projective coordinates:

$$ \fonctionbis{\mathbb{RP}^1}{ [-\frac{\pi}{2}, \frac{\pi}{2} ]}{ \left[  \begin{pmatrix}
      x \\ y
    \end{pmatrix} \right] }{ \arctan \left( \frac{x}{y} \right)} $$

\noindent At the level of the Veech group it means projecting it to
$\mathrm{PSL_2}(\mathbb{R})$ by the canonical projection $\pi$. Let us
denote by
$\veechsg \subset \pi( \discoveech) \subset
\mathrm{PSL}_2(\mathbb{R})$ the group generated by the two elements:

$$ \veechsg :=
\left< A =
  \begin{pmatrix}
    1& 6 \\
    0 & 1
  \end{pmatrix}
  \hspace{0.3cm}, B =
  \begin{pmatrix}
    1& 0 \\
    \frac{3}{2} & 1
  \end{pmatrix}
\right> $$

\noindent We will study the group $\veechsg$ as a Fuchsian group, that
is a \textbf{discrete} group of isometries of the real hyperbolic
plane $\mathbb{H}^2$. For the action of a Fuchsian group $\Phi$ on
$\mathbb{RP}^1$, there are two invariant subsets which we will
distinguish:

\begin{itemize}

\item one called its \textit{limit set} on which $\Phi$ acts minimally
  and that we will denote by $\Lambda_{\Phi} \subset \mathbb{RP}^1$;

\item the complement of $\Lambda_{\Phi}$ which is called its
  \textit{discontinuity set}, on which $\Phi$ acts properly and
  discontinuously and which we will denote by $\Omega_{\Phi}$.

\end{itemize} We will give precise definitions in Section \ref{subsec
  the limit set }. In restriction to the discontinuity set, one can
form the quotient by the action of the group. The topological space
$\quotient{\Omega_{\Phi}}{\Phi}$ is a manifold of dimension one: a
collection of real lines and circles. \newline
\noindent We will show in Proposition \ref{proposition discontinuity
  set} that for the group $\Gamma$ this set is a single circle, and
therefore a fundamental domain $I$ for the action of the group
$\veechsg$ can be taken to be a single interval (we will make it
explicit:
$I = [\arctan(\frac{1}{4}),\arctan(1)] \subset [- \frac{\pi}{2},\frac{\pi}{2}[ \simeq \mathbb{RP}^1 $).
The dynamic of the directional foliations in the directions $\theta$
belonging to the interior of the interval $I$ will be studied in
Section \ref{secgendir}.

\begin{remark}
  We will prove in Section \ref{secglopic} that the group $\veechsg$
  is actually equal to the full Veech group of the surface $\disco$.
\end{remark}

\subsection{The action of the group $\veechsg$ on $\Hy$.}
\label{subsecfucgam}

Two hyperbolic isometries $A,B \in \mathrm{Isom_+(\mathbb{H}^2)}$ are
said to be in \textbf{Schottky position} if the following condition
holds:\\
\begin{center}

  There exists four disjoints domains $D_i$, $ 1 \le i \le 4$ which
  satisfy
$$ A ( D_1^c) = D_2 \hspace{2 cm } B ( D_3^c) = D_4 $$
where $D_i^c$ denotes the complementary set of $D_i$.

\end{center}

\vspace{2mm}
\noindent A group generated by two elements in Schottky position is
also called a \textbf{Schottky group}. Figure \ref{figthrsph} illustrates this situation.

\begin{proposition}
\label{fundomain}
  The group $\veechsg$ is Schottky. Moreover the surface
  $M_{\veechsg}$ is a three punctured sphere with two cusps and one
  end of infinite volume.
\end{proposition}

\begin{proof}
  Viewed in the upper half plane model of $\mathbb{H}^2$ the action is
  easily shown to be Schottky. In fact the action of $A$ \resp $B$)
  becomes $ z \mapsto z + 6 $ \resp
  $z \mapsto \frac{z}{\frac{3z}{2} + 1}$).  The two matrices are
  parabolic and fix $\infty$ and $0$ respectively.  Moreover, we
  observe that $A(-3) = 3$ and $B(-1)=2$.  Figure \ref{figthrsph}
  below shows that the two matrices $A$ and $B$ are in Schottky
  position with associated domains $D_i$ for $1 \le i \le 4$.

    \begin{figure}[h!]
      \hspace*{1cm}
      \begin{subfigure}{.45\linewidth}
        \hspace{-1cm}
        \includegraphics[height=4cm]{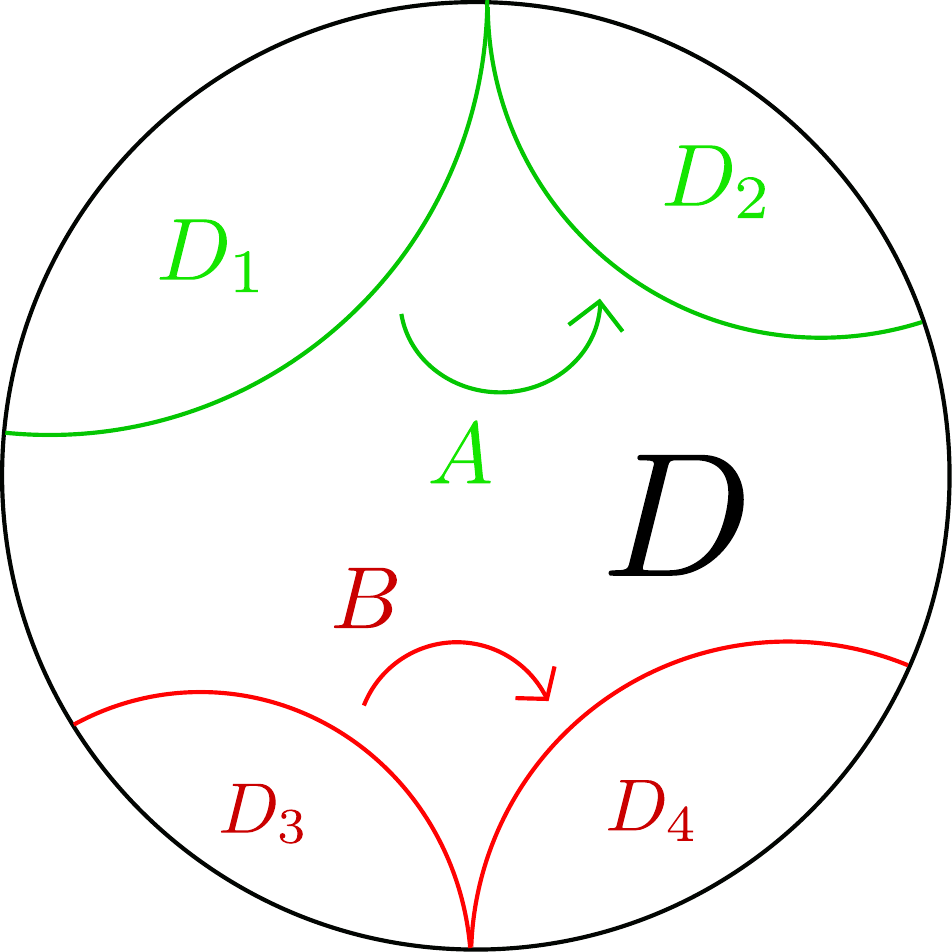}
      \end{subfigure}
      \begin{subfigure}{.4\linewidth}
        \centering
        \includegraphics[height=4cm]{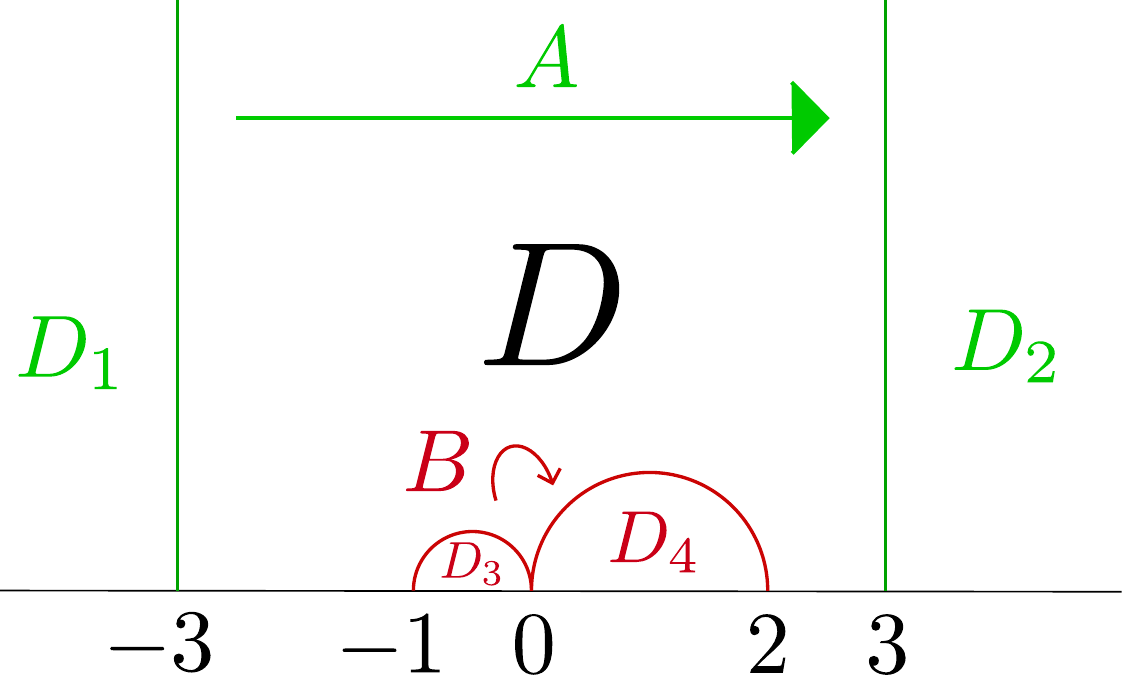}
      \end{subfigure}
      \caption{A fundamental domain for the action of the group
        $\veechsg$ acting on the hyperbolic plane.}
      \label{figthrsph}
    \end{figure}

   \noindent The domain $D$ of Figure \ref{figthrsph} is a fundamental domain
    for the action of $\veechsg$. The isometry $A$ identifies the two
    green (light) boundaries of $D$ together and $B$ the two red (dark) ones. The
    quotient surface is homeomorphic to a three punctured sphere.
  \end{proof}

  \subsection{The limit set and the discontinuity set.}
  \label{subsec the limit set }

  The following notion will play a key role in our analysis of the
  affine dynamics of the surface $\disco$:

\begin{definition}
  The \textbf{limit set} $\Lambda_{\Phi} \subset S^1$ of a Fuchsian
  group $\Phi$ is the set of accumulation points in
  $\mathbb{D} \cup S^1$ of any orbit $\Phi \cdot \{z_0\}$,
  $z_0 \in \mathbb{D}$ where $\mathbb{D} \subset \mathbb{C}$ is the
  disk model for the hyperbolic plane $\mathbb{H}^2$.
\end{definition}

\noindent The complementary set of the limit set is the good tool to
understand the infinite volume part of such a surface.

\begin{definition}
  The complementary set
  $\Omega_{\Phi} := S^1 \backslash \Lambda_{\Phi}$ is by definition the \textbf{set
    of discontinuity} of the action of $\veechsg$ on the circle.
\end{definition}

\noindent The group $\Phi$ acts properly and discontinuously on the
set of discontinuity. One can thus form the quotient space
$\quotient{\Omega_{\Phi}}{\Phi}$ which is a manifold of dimension one
: a collection of circles and real lines.
\noindent These sets are very well understood for Schottky groups
thanks to the ping-pong Lemma, for further details and developments
see \cite{dalhorgeo} chapter 4.

\begin{proposition}[ping-pong lemma]
  \label{prosch}
A Schottky group $\Phi$ is freely generated by any two elements in Schottky position. Moreover
 \begin{itemize}
  \item if the quotient surface $\quotient{\mathbb{H}^2}{\Phi}$ is of finite volume, the limit set is the full circle;
  \item otherwise, the limit set is homomorphic to a Cantor set.
 \end{itemize}
\end{proposition}

\begin{remark}
	The third item case is the one of interest of this article.
\end{remark}

\noindent The following theorem will be used in \ref{secglopic} to
prove our main Theorem \ref{thm1}
\begin{theorem}[Ahlfors, \cite{Ahlfors}]
  \label{thmahl}
  A finitely generated Fuchsian group satisfies the following
  alternative:
  \begin{enumerate}
  \item either its limit set is the full circle $S^1$ ;
  \item or its limit set is of zero Lebesgues measure
  \end{enumerate}
\end{theorem}

\noindent In our setting, it is clear that the limit set is not the
full circle, thus the theorem implies that the limit set of $\Gamma$
is of zero Lebesgue measure.

\subsection{The action on the discontinuity set and the fundamental
  interval}

\noindent The following proposition is the ultimate goal of this section.

\begin{proposition}
  \label{proposition discontinuity set}
  The quotient space $$\quotient{\Omega_\veechsg}{\veechsg}$$ is a
  circle. A fundamental domain for the action of $\veechsg$ on
  $\Omega_{\veechsg}$ corresponds to the interval of slopes
  $I =[\arctan \left( \frac{1}{4} \right),\arctan(1)]$.
\end{proposition}

\noindent Foliations defined by slopes which belong to this precise interval
will be studied in section \ref{secgendir}. To prove this proposition
we will use the associated hyperbolic surface $M_{\veechsg}$ and link
its geometrical and topological properties to the action of the group
$\veechsg$ on the circle. \newline

\noindent The definition of the limit set itself implies that it is invariant by
the Fuchsian group. One can therefore seek a geometric interpretation
of such a set on the quotient surface. We will consider the smallest
convex set (for the hyperbolic metric) which contains all the
geodesics which start and end in the limit set
$\Lambda_{\veechsg}$. We denote it by $C(\Lambda_{\veechsg})$. Because
the group $\veechsg$ is a group of isometries it preserves
$C(\Lambda_{\veechsg})$.

\begin{definition}
  \label{definition convex core}
  The \textbf{convex core} of a hyperbolic surface $M_{\veechsg}$,
  denoted by $C(M_{\veechsg})$ is defined as:
  $$ C(M_{\veechsg}) := \quotient{C(\Lambda_{\veechsg})}{\veechsg} $$
\end{definition}

\noindent As a quotient of a $\veechsg$-invariant subset of $\mathbb{H}^2$, $C(M_{\veechsg})$ is a subset of the surface $M_{\Gamma}$. The convex core of a
Fuchsian group is a surface with geodesic boundary, moreover if the
group is finitely generated the convex core has to be of finite
volume. As a remark, a Fuchsian group is a lattice if and only if we
have the equality $C(M_{\veechsg}) = M_{\veechsg}$. In the special
case of the Schottky group $\veechsg$ the convex core is a surface
whose boundary is a single closed geodesic as it is shown on Figure
$\ref{bout}$. For a finitely generated group we will see that we have
a one to one correspondence between connected components of the
boundary of the convex core and connected component of the quotient of
the discontinuity set by the group.

\begin{figure}[h!]
  \includegraphics[width=.9\linewidth]{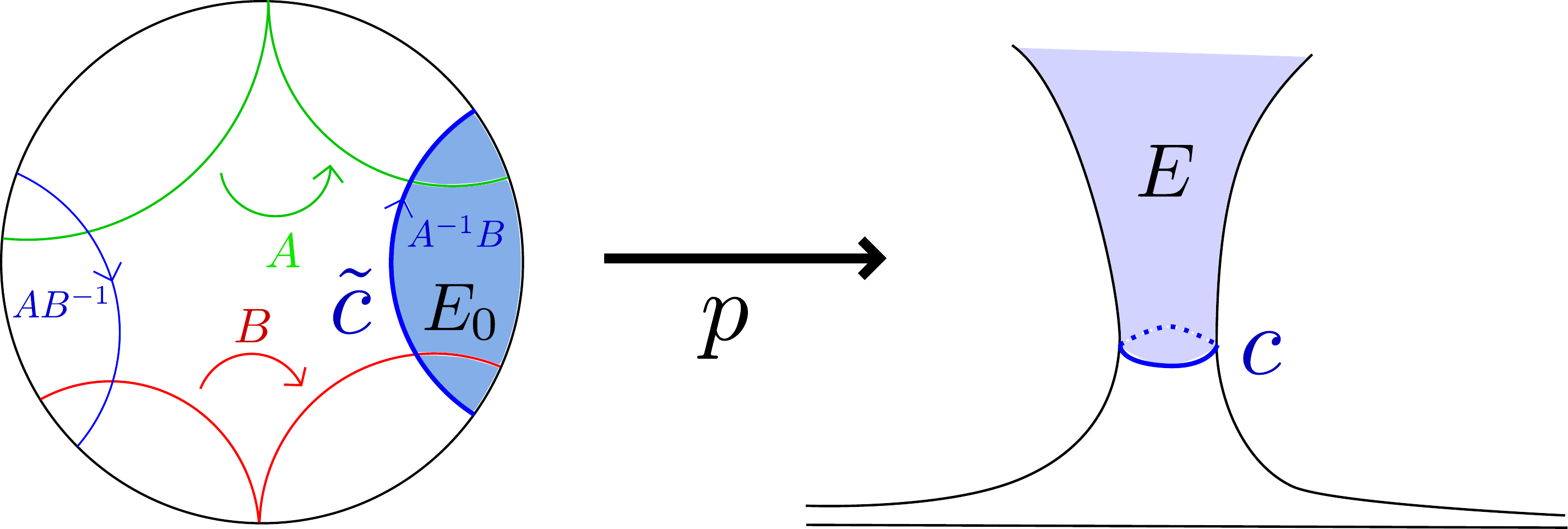}
  \caption{The closed geodesic $c$ cuts the surface $M_{\veechsg}$
    into two pieces. The colored part is the only infinite volume end and
    its complementary is the convex core. The choice of a lift
    $\tilde{c}$ of the geodesic $c$ made on the picture allows us to
    describe the isometry which translate along $\tilde{c}$ in terms
    of the generators of the group $\veechsg$. Indeed the pairing of
    the edges of the fundamental domain given by the action of the
    group $\veechsg$ shows that the geodesic $\tilde{c}$ is the
    translation axes of the matrix $A^{-1}B$.}
  \label{bout}
\end{figure}

The following lemma is the precise formulation of what we discussed
above
\begin{lemma}
  \label{pro bord cantor}
  Let $\veechsg$ a finitely generated Fuchsian group. Any connected
  component $I_0$ of the discontinuity set $\Omega_{\veechsg}$ is
  stabilised by a cyclic group generated by a hyperbolic isometry
  $\gamma_0$. Moreover $ \partial I_0 $ is composed of the two fixed
  points of the isometry $\gamma_0$.
\end{lemma}

\noindent We will keep notations introduced with Figure \ref{bout}. We start
by showing that for any choice of a lift $\tilde{c}$ in the universal
cover of a geodesic $c$ in the boundary of the convex core one can
associate an isometry verifying the properties of Lemma
$\ref{pro bord cantor}$.  Let $c$ be a closed geodesic consisting of a
connected component of the boundary of the convex core of
$M_{\veechsg}$. One can choose a lift $\tilde{c}$ of such a geodesic
in the universal cover, the geodesic $\tilde{c}$ is the axis of some
hyperbolic isometry $\Phi$, whose fixed points are precisely the
intersection of $\tilde{c}$ with the circle. As an element of the
boundary of $C(M_\veechsg)$ it cuts the surface $M_{\veechsg}$ into two
pieces: $C(M_\veechsg)$ and an end $E$. One can check that this
isometry $\Phi$ is exactly the stabiliser of the connected component
$E_0$ of $p^{-1}(E)$ whose boundary is the geodesic
$\tilde{c}$. Therefore such an isometry stabilises the connected
component of the discontinuity set given by the endpoints of the
geodesic $\tilde{c}$. We have shown that given a boundary component of
$C(M_\veechsg)$ one can associate an element (in fact a conjugacy
class) of the group $\veechsg$ which stabilises a connected component
of $\Omega_{\Gamma}$. We will not show how to associate a geodesic in
the boundary of the convex core to a connected component of the
discontinuity set.

\begin{remark}
We want to put the emphasis on the fact the
assumption that the group is finitely generated will be used here. The
key point is the geometric finiteness theorem \cite[Theorem
4.6.1]{booksvetlana} which asserts that any finitely generated group
is also \textbf{geometrically finite}.  It means that the action of
such a group admits a polygonal fundamental domain with finitely many
edges. It is not difficult to exhibit from such a fundamental domain
the desired geodesic by looking at the pairing induced by the group,
as it is done in Figure $\ref{bout}$ for our Schottky group.
\end{remark}

\begin{corollary}
  \label{cor bouts compo}
  Connected components of \ $\quotient{\Omega_{\veechsg}}{\veechsg}$
  are in one to one correspondence with infinite volume ends of the
  surface $\quotient{\mathbb{H}^2}{\veechsg}$.
\end{corollary}

\noindent We now have all the material needed to prove Proposition \ref{proposition discontinuity set}.

\begin{proof}[Proof of Proposition \ref{proposition discontinuity
    set}]

  Because the surface $\quotient{\mathbb{H}^2}{\veechsg}$ has only one
  end of infinite volume Corollary \ref{cor bouts compo} gives
  immediately that $\quotient{\Omega_\veechsg}{\veechsg}$ is a single
  circle.  Proof of the second part of Proposition \ref{proposition discontinuity set} consists in a simple matrix computation. Figure \ref{bout}
  gives explicitly the elements of the group $\veechsg$ which
  stabilise a connected component of the discontinuity set. We then
  have to prove:

$$	 \left[ AB^{-1}\begin{pmatrix} 1 \\ 1	\end{pmatrix} \right] = \left[ \begin{pmatrix} 4 \\  1
  \end{pmatrix} \right] $$ where $ [ X ]$ is the projective class of
the vector $X$. The computation is easy:

\begin{align*}
  AB^{-1}\begin{pmatrix}
    1 \\ 1
  \end{pmatrix} & = \begin{pmatrix} -8 & 6 \\ - \frac{3}{2} & 1
  \end{pmatrix} \begin{pmatrix} 1 \\ 1
  \end{pmatrix} \\
                & = \begin{pmatrix} -2 \\ - \frac{1}{2}
                \end{pmatrix} \\
\end{align*}

\end{proof}

\section{Generic directions and Rauzy induction.}
\label{secgendir}

\noindent The boundary of $\Hy$ is canonically identified with $\mathbb{RP}^1$
through the natural embedding $\Hy \rightarrow \mathbb{CP}^1$.  Recall
that the action of $\mathrm{PGL}(2,\C)$ by Möbius transformations on
$\C$ is induced by matrix multiplication on $\mathbb{CP}^1$ after
identification with $\C$ by the dilatation chart $z \rightarrow
[z:1]$. Thus the action of matrices of the Veech group on the set of
directions corresponds to
the action of these matrices as homographies on the boundary of $\Hy$.\\

\begin{figure}[h!]
  \centering
  \includegraphics[width=.5\linewidth]{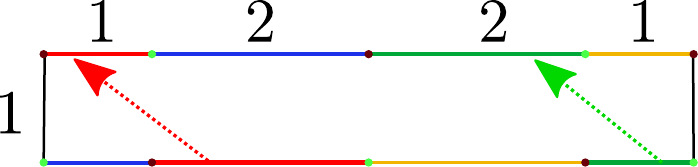}
  \caption{Attractive leaf on the left, repulsive on the right}
  \label{attractive}
\end{figure}

\noindent We notice straight away that for directions $[t:1]$ with $t$ between
$1$ and $2$, there is an obvious attractive leaf of dilatation
parameter $1/2$ (see Figure \ref{attractive}).
There is also a repulsive closed leaf in this
direction. This will always be the case since $-\mathrm{Id}$ is in the Veech
group, sending attractive closed leaves to repulsive closed leaves.

\noindent In the following we will describe dynamics of the directional
foliation for $t$ between $2$ and $4$. According to Section
\ref{subsecgroup} the interval of direction $[1,4]$ is a fundamental
domain for the action of $\veechsg$ on $\Omega_{\veechsg}$ its
discontinuity set. Moreover this discontinuity set has full Lebesgue
measure in the set of directions thus understanding the dynamical
behaviour of a typical direction therefore amounts to understanding it
for $t \in [1,4]$. Further discussion on what happens in other
directions will be done in the next section.\\

\subsection{Reduction to an AI}
\begin{figure}[h!]
  \centering
  \includegraphics[width=.5\linewidth]{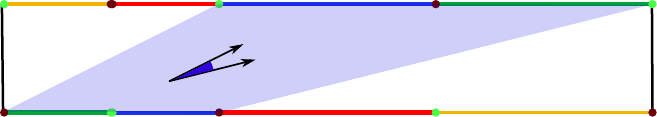}
  \caption{The stable subsurface.}
  \label{stable}
\end{figure}

The directions for $t \in [2,4]$ have an appreciable property; they
correspond to the directions of a subsurface invariant under the (oriented) foliation represented in Figure \ref{stable}. Every leaf in
the given angular set of directions that enters the subsurface will
stay trapped in it thereafter.  We therefore seek attractive closed leaves in
this subset. To do so, take a horizontal interval joining the boundary components of this invariant subsurface and consider the first return map on it. It has a specific form (close to an
affine interval exchange) which we will study in this section.\\

\newcommand{\IetFam}[2]{\mathcal I \left(#1, #2 \right)}

In the following, we use the notation AI to talk of a piecewise affine injection on an interval.
For any $m,n \in \mathbb N$, let $\IetFam m n$ be the set of AIs
defined on $[0,1]$ with two intervals on which it is affine and such
that the image of the left interval is an interval of its length
divided by $2^{n}$ which rightmost point is $1$, and that the image of
its right interval is an interval of its length divided by $2^{m}$
which leftmost point is 0 (see Figure \ref{iet} for such an AI
defined on $[0,1]$).  When representing an AI, we will color the
intervals on which it is affine in different colors, and represent a
second interval on which we color the image of each interval with the
corresponding color; this will be sufficient to characterise the
map. The geometric representation motivates the fact that we call the
former and latter sets of intervals the \textit{top} and \textit{bottom
  intervals}.

\begin{figure}[h]
  \centering
  \includegraphics[width=.5\linewidth]{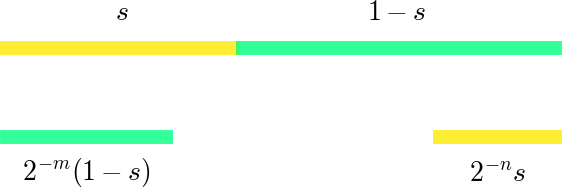}
  \caption{Geometric representation of an element of $\IetFam m n$}
  \label{iet}
\end{figure}

\noindent Note that the cross-sections defined on the subsurface of Figure \ref{stable} are in $\IetFam 1
1$. We will study the dynamical behaviour of this family of AI.\\

\subsection{Rauzy-Veech induction}

Let $T$ be an AI and $D$ be its interval of definition.  The first
return map on a subinterval $D' \subset D$, $T': D' \rightarrow D'$ is
defined for every $x \in D'$, as
$$ T'(x) = T^{n_0}(x) \text{ where } n_0 = \inf \{ n \geq 1 \ | \T^n(x) \in D' \}$$

Since we have no information on the recurrence properties of an AI this first return map is a priori not defined on an arbitrary
sub-interval. Nonetheless generalizing a wonderful algorithm of Rauzy
\cite{Rauzy} for IETs, we get a family of subinterval on which this
first return map is well defined. Associating to an AI its first return
map on this well-chosen smaller interval will be called the
Rauzy-Veech induction.

The general idea in the choice of this interval is to consider the
smallest of the top and bottom intervals at one end of $D$ (left
or right) the interval of definition.  We then consider the first
return map on $D$ minus this interval.\\

In the following we describe explicitly the induction for the simple
family $\IetFam m n$. A general and rigorous definition of Rauzy-Veech
induction in the more general context of both AIs and AIETs is certainly possible with a lot of interesting
questions emerging but is beyond the scope of this article.\\

Assume now that $T$ is an element of $\IetFam m n$, let
$A, B \subset D$ be the left and right top intervals of $T$, and
$\lambda_A, \lambda_B$ their length.
Several distinct cases can happen, \\

\begin{enumerate}
\item
  \label{continue}
  \begin{enumerate}
  \item $B \varsubsetneq T(A)$ \ \textit{i.e.} \
    $\lambda_B < 2^{-n} \lambda_A$.

  \begin{figure}[h!]
    \hspace*{1cm}
    \begin{subfigure}{.45\linewidth}
      \centering
      \includegraphics[width=.9\linewidth]{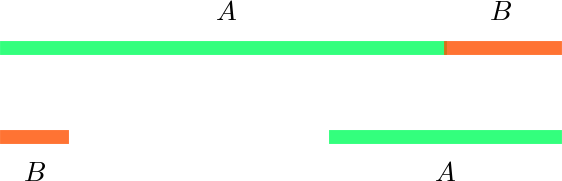}
      \subcaption{Example of such AI}
    \end{subfigure}
    \begin{subfigure}{.45\linewidth}
      \centering
      \includegraphics[width=.9\linewidth]{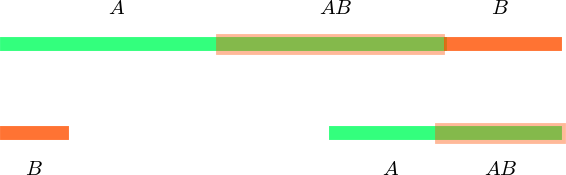}
      \subcaption{Right Rauzy-Veech induction}
    \end{subfigure}
  \end{figure}

  We consider the first return map on $D' = D - B$.  $T^{-1}(B)$ of
  length $2^n \lambda_B$ has no direct image by $T$ in $D'$ but
  $T(B) \subset D'$. Thus for the first return map, this interval will
  be sent directly to $T(B)$ dividing its length by $2^{n+m}$.  We
  call this a right Rauzy-Veech induction of our AI.  The new AI
  is in $\IetFam {m+n} {n}$, and its length vector
  $(\lambda_A', \lambda_B')$ satisfies
  $$
  \left(
    \begin{array}{c}
      \lambda_A' \\
      \lambda_B'
    \end{array}
  \right) = \overbrace{ \left(
      \begin{array}{c c}
        1 & -2^n \\
        0 &  2^n
      \end{array}
    \right) }^{R_{m,n}} \left(
    \begin{array}{c}
      \lambda_A \\
      \lambda_B
    \end{array}
  \right)
  $$
  \label{bad}

\item If $A \varsubsetneq T(B)$ \ \textit{i.e.} \
  $2^{-m} \lambda_B > \lambda_A$.

    \begin{figure}[h!]
      \hspace*{1cm}
      \begin{subfigure}{.45\linewidth}
        \centering \hspace*{.5cm}
        \includegraphics[width=.9\linewidth]{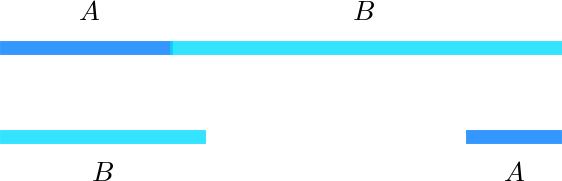}
        \subcaption{AI satisfying $A \subset T(B)$}
      \end{subfigure}
      \begin{subfigure}{.45\linewidth}
        \centering \hspace*{.5cm}
        \includegraphics[width=.9\linewidth]{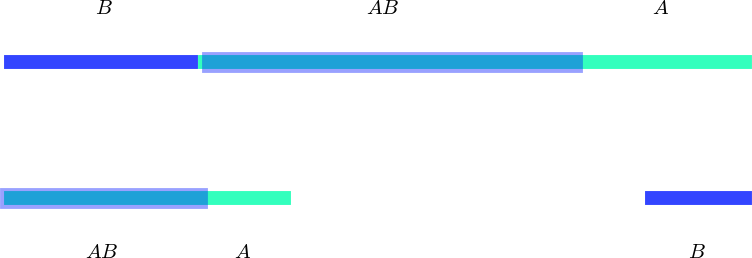}
        \subcaption{Right Rauzy-Veech induction}
      \end{subfigure}
    \end{figure}

    In this case, the right Rauzy-Veech induction is not well-defined therefore we consider
    the first return map on $D' = D-A$ which we call the left
    Rauzy-Veech induction of our AI. We obtain a new AI in
    $\IetFam {m} {n+m}$ and its length vector
    $(\lambda_A', \lambda_B')$ satisfies,
  $$
  \left(
    \begin{array}{c}
      \lambda_A' \\
      \lambda_B'
    \end{array}
  \right) = \overbrace{ \left(
      \begin{array}{c c}
        2^m & 0 \\
        -2^m & 1
      \end{array}
    \right)}^{L_{m,n}} \left(
    \begin{array}{c}
      \lambda_A \\
      \lambda_B
    \end{array}
  \right)
  $$\\

\end{enumerate}

\noindent Note that the two subcases presented above are mutually exclusive since the considered maps are strictly contracting.

\item $T(A) \subset B$ \ \textit{i.e.} \ $\lambda_B \geq 2^{-n} \lambda_A$ and $T(B) \subset A$ \ \textit{i.e.} \ $2^{-m} \lambda_B \leq \lambda_A$.\\
  \label{fixed point}

  We consider the first return map on the subinterval $D' = D - T(A)$.
  Then $A$ has no direct image by $T$ in $D'$ but
  $T^2(A) \subset T(B) \subset D'$.  Thus in the first return map,
  this interval will
  be sent directly to $T^2(A)$ dividing its length by $2^{n+m}$.\\

    \begin{figure}[h!]
      \hspace*{1cm}
      \begin{subfigure}{.45\linewidth}
        \centering \hspace*{.5cm}
        \includegraphics[width=.9\linewidth]{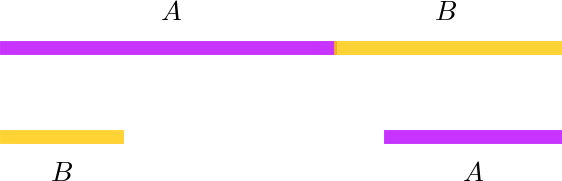}
        \subcaption{AI satisfying $T(B) \subset A$}
      \end{subfigure}
      \begin{subfigure}{.45\linewidth}
        \centering \hspace*{.5cm}
        \includegraphics[width=.9\linewidth]{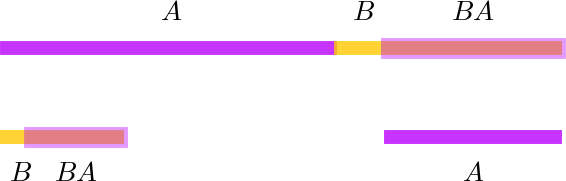}
        \subcaption{Right Rauzy-Veech induction}
      \end{subfigure}
    \end{figure}

    \noindent Then $T^2(A) \subset A$ thus the induced map has an attractive fixed
    point of
    derivative $2^{-n-m}$.\\

  \end{enumerate}

  \begin{remark}
    \label{set}
    The set of length for which we apply left or right Rauzy-Veech
    induction in the above trichotomy is exactly the set on which
    lengths $\lambda_A'$ and $\lambda_B'$ implied by the above
    formulas are both positive.\\

    \noindent More precisely,
    $0 \leq \lambda_B \leq 2^{-n} \lambda_A \iff R_{m,n} \cdot \left(
      \begin{array}{c}
        \lambda_A \\
        \lambda_B
      \end{array}
    \right) \geq 0$,\\
    and
    $0 \leq 2^{-m} \lambda_B \leq \lambda_A \iff L_{m,n} \cdot \left(
      \begin{array}{c}
        \lambda_A \\
        \lambda_B
      \end{array}
    \right) \geq 0$.\\

    \noindent This will be useful later on to describe the set of
    parameters which corresponds to the sequence of induction moves we
    apply.
  \end{remark}

  \paragraph*{\bf The algorithm.}

  We define in what follows an algorithm based on Rauzy induction
  that will allow us to determine if an element of $\IetFam 1 1$ has
  an attractive periodic orbit; and if so the length of its periodic
  orbit (or equivalently the dilatation coefficient of the
  associated leaf in $\disco$).\\

  \noindent The algorithm goes the following way:\\

  The entry is an element of $\IetFam m n$,
  \begin{enumerate}
  \item If the entry is in case $(1)$, perform in case $(a)$ the right
    Rauzy induction $R$ or in case $(b)$ the left Rauzy induction $L$
    to obtain an element of $\IetFam {m+n} {n}$ or $\IetFam {m} {n+m}$
    respectively.  Repeat the loop with this new element.\\

  \item If it is in case $(2)$, it means that the first return map on
    a well chosen interval has a periodic attractive point of
    derivative $2^{-m-n}$. The algorithm
    stops.\\

  \end{enumerate}

  Alongside the procedure comes a sequence of symbols $R$ and $L$
  keeping track of whether we have performed the Rauzy induction on
  the left or on the right at the $n^{th}$ stage. This sequence is
  finite if and only if the algorithm described above finishes. An
  interesting phenomenon will happen for AI for which the induction
  never stops, and will be described latter.

  \subsection{Directions with attractive closed leaf}
\label{subsec attractive}  
  In the directions of Figure \ref{stable} corresponding to parameters in $[2,4]$ in projective coordinates, we consider the first
  return map of the directional foliation on the interval given by the
  two length 1 horizontal interval at the bottom of the rectangle.  We
  have chosen directions such that the first return map is well
  defined although it is not bijective, and it belongs to
  $\IetFam 1 1$.  The ratio of the two top intervals' length will vary
  smoothly between $0$ and $\infty$ depending on the direction we
  choose. We parametrise this family of AI by $s \in I := [0,1]$,
  where $(s, 1-s)$ is the length vector of the element of
  $\IetFam 1 1$ we get.  The purpose of this section is to characterise the
  subspace $H \subset I$ for which the above algorithm stops, in
  particular they correspond to AI with a periodic orbit. The case
  of $I-H$ will be settled in the next subsection.\\

  \noindent We describe for any finite word in the alphabet $\{ L, R \}$,
  $w = w_1 \dots w_{l-1}$, the subset of parameters
  $H(w) \subset H \subset I$ for which the algorithm stops after the
  sequence $w$ of Rauzy-Veech induction moves.\\

  \noindent We associate to $w$ the sequences $n_1=1, \dots, n_{l}$,
  $m_1=1, \dots, m_{l}$ and $M_1=Id, \dots, M_{l}$ defined by the
  recursive properties,
$$
m_{i+1} =
\begin{cases}
  m_i       & \text{if } w_i = L\\
  n_i + m_i & \text{if } w_i = R
\end{cases}
, \hspace*{.5cm} n_{i+1} =
\begin{cases}
  n_i + m_i & \text{if } w_i = L\\
  n_i & \text{if } w_i = R
\end{cases}
,$$
$$
M_{i+1} =
\begin{cases}
  L_{m_i, n_i} \cdot M_i & \text{if } w_i = L\\
  R_{m_i, n_i} \cdot M_i & \text{if } w_i = R
\end{cases}$$

\noindent Let $s \in I$ such that we can apply Rauzy-Veech inductions
corresponding to $w$ to the element of $\IetFam 1 1$ of lengths
$(s, 1-s)$. The induced AI after all the steps of the induction is
in $\IetFam {m_{l}} {n_{l}}$ and its length vector is
$$M_{l} \cdot
\left(
  \begin{array}{c}
    s   \\
    1-s \\
  \end{array}
\right) := \left(
  \begin{array}{c c}
    a & b \\
    c & d \\
  \end{array}
\right) \cdot \left(
  \begin{array}{c}
    s   \\
    1-s \\
  \end{array}
\right) = \left(
  \begin{array}{c}
    (a-b)s + b \\
    (c-d)s + d \\
  \end{array}
\right)
$$

\noindent Following Remark \ref{set}, the property of $s$ being such that we can
apply all the Rauzy-Veech inductions corresponding to $w$ to the
initial AI in $\IetFam 1 1$ is equivalent to $(a-b)s + b \geq 0$ and
$(c-d)s + d \geq 0$.  An induction on $M_i$ shows that it is an
integer matrix with $a, d \geq 0$, $b,c \leq 0$, hence
$s \in \left[ \frac {-b}{a-b}, \frac {d}{d-c} \right] =: I(w)$.
$H(w)$ will be the central subinterval of $I(w)$ for which the
induced AI in $\IetFam {m_l} {n_l}$ is in case $(2)$.\\

\noindent Consider the sets
$$H_k := \bigcup_{\left|w\right| \leq k} H(w) \ \text{ and } \ H = \bigcup_{k} H_k.$$
Notice that $H$ has the same construction as the complement of the
Cantor triadic set; each $H_k$ is constructed from $H_{k-1}$ by adding
an interval in the interior of each interval which is a connected component of
$I - \bigcup_{j < k} H_j$.\\

\noindent The rest of the subsection aims now at proving the following lemma,

\begin{lemma}
\label{lemmecarlos}
  $H \subset I$ has full Lebesgue measure.
\end{lemma}

\noindent As a preliminary we need the following Lemma which will be used later on in the proof.

\begin{lemma}
\label{lemmeloscar}
  For any word $w$ in $\{ R, L \}$, if $M(w) = \left(
    \begin{array}{c c}
      a & b \\
      c & d \\
    \end{array}
  \right)$, we have, $$2^{-1} \leq x = \frac{a-b}{d-c} \leq 2 $$
\end{lemma}

\begin{proof}
  The proof goes by induction on the length of $w$. Let us assume that
  $2^{-1} \leq \frac{a-b}{d-c} \leq 2$ for some $w$.  We denote by

    $$\left(
    \begin{array}{c c}
      a' & b' \\
      c' & d' \\
    \end{array}
  \right) = R_{m,n} \cdot \left(
    \begin{array}{c c}
      a & b \\
      c & d \\
    \end{array}
  \right) = \left(
    \begin{array}{c c}
      a-2^n c & b - 2^n d \\
      2^n c   & 2^n d     \\
    \end{array}
  \right)$$

  \noindent Thus $\frac{a'-b'}{d'-c'} = 2^{-n} \frac{a-b}{d-c} + 1$
  from which the inequality follows.

  $$\left(
    \begin{array}{c c}
      a' & b' \\
      c' & d' \\
    \end{array}
  \right) = L_{m,n} \cdot \left(
    \begin{array}{c c}
      a & b \\
      c & d \\
    \end{array}
  \right) = \left(
    \begin{array}{c c}
      2^m a & 2^m b \\
      c - 2^m a & d - 2^mb \\
    \end{array}
  \right)$$

  \noindent The inequality is similar to the previous one.
\end{proof}

\begin{proof}[Proof of Lemma \ref{lemmecarlos}]
  We will prove in the following that for any non-empty word $w$,
  \begin{equation}
    \label{cantorin}
    \frac {|H(w)|}{|I(w)|} \geq \delta
  \end{equation}
  for some $\delta > 0$.  Thus at each step $k$, $H_k$ is at least a
  $\delta$-proportion larger in Lebesgue measure than $H_{k-1}$.  This
  implies the Lemma because
  $$\lambda(H) \geq \lambda(H_k) \geq 1 - (1-\delta)^k
  \text{ for any } k$$

\noindent
We now show Inequality \ref{cantorin}. Let $w$ be any finite word in
the alphabet $\{ L, R \}$. For convenience we normalise the interval
$I(w)$ for such that it is $[0,1]$. We denote by
$\left(\lambda_A(s),\lambda_B(s)\right)$ the length vector of the AI
induced by the sequence $w$ of Rauzy-Veech inductions.  These two
lengths are linear functions of $s$, $\lambda_A$ is zero at the left
end of the interval and $\lambda_B$ is zero at the right end.  As a
consequence, these two functions have the form
$\lambda_A(s) = \alpha s$ and $\lambda_B(s) = \beta(1-s)$ for
$s \in [0,1]$, where $\alpha$ and $\beta$ are the maximal values of
$\lambda_A$ and $\lambda_B$ respectively equal to according to the
previous computations
$$\alpha = (a-b) \frac d {d-c} +b = \frac{ad-bd+bd-bc}{d-c} = \frac{det(M_l)}{d-c}$$
and
$$\beta = (c-d)\frac {-b}{a-b} + d = \frac{-bc+bd+da-db}{a-b} = \frac{det(M_l)}{a-b}$$
We see that
$\lambda_A(s) = 2^{-m} \lambda_B(s) \iff 2^m \alpha s = \beta (1 -s)
\iff s = \frac{\beta}{2^m \alpha + \beta}$ and similarly
$\lambda_B(s) = 2^{-n} \lambda_A(s) \iff 2^n \beta(1-s) = \alpha s
\iff s = \frac {2^n \beta}{\alpha + 2^n \beta}$.  Hence
$$\lambda_A(s) \leq 2^{-m} \lambda_B(s) \iff s \in \left[0, \frac \beta {2^m \alpha + \beta}\right]$$
and
$$\lambda_B(s) \leq 2^{-n} \lambda_A(s) \iff s \in \left[ \frac {\beta}{2^{-n}\alpha + \beta}, 1 \right]$$
thus
$$H(w) = \left[ \frac \beta {2^m \alpha + \beta}, \frac {\beta} {2^{-n}\alpha + \beta} \right].$$
If we denote by $x = \frac \alpha \beta = \frac {a-b}{d-c}$,
$$ \frac {\norm{H(w)}}{\norm{I(w)}} = \frac {1} {1 + 2^{-n}x} - \frac 1 {1 + 2^m x}$$\\

\noindent Lemma \ref{lemmeloscar} implies directly that
$$ \frac {\norm{H(w)}}{\norm{I(w)}} \geq  \frac {1} {1 + 2^{-n+1}} - \frac 1 {1 + 2^{m-1}}$$

\noindent Hence for $w$ not empty, either $n \geq 2$ or $m \geq 2$
thus either
$$\frac {\norm{H(w)}}{\norm{I(w)}} \geq \frac 2 3 - \frac 1 2 = \frac 1 6 \text{ or }
\frac {\norm{H(w)}}{\norm{I(w)}} \geq \frac 1 2 - \frac 1 3 = \frac 1 6$$

\end{proof}

\subsection{AI with infinite Rauzy-Veech induction}
We focus in this subsection on what happens for AIs on which we
apply Rauzy-Veech induction infinitely many times.  First, remark that
if we apply the induction on the same side infinitely many times, the
length of the top interval of the corresponding side on the induced
AI is multiplied each time by a positive power of $2$, therefore it
goes to infinity. Yet the total length of the subinterval is bounded
by $1$ the length of the definition interval from which we started the
induction. Thus the length of the interval has to be zero; this
corresponds to the case where there is an attractive saddle connection, considered as a part of the cases treated above since we chose to take $H(w)$
closed. \\

In consequence, for an AI $T$ with parameter in $I - H$, we apply
Rauzy-Veech induction infinitely many times, and the sequence of
inductions we apply is not constant after a finite number of
steps. Now let as above $D$ be the interval of definition of the given
AI, and $A, B \subset D$ be its two intervals of
continuity. Remark that the induction keeps the right end of $T(B)$
and the left end of $T(A)$ unchanged. Moreover the induction divides
the length of one of the bottom interval (depending on which
Rauzy-Veech induction we apply) by at least two because we iterate maps whose dilatation factor is at most $\frac{1}{2}$. In turn if
we consider $I_n$ to be the open subinterval of $D$ on which we
consider the first return map after the $n$-th induction, the limit of
these nested intervals is
$$I_\infty := \bigcap_{n=1}^\infty I_n = D- T(A) \cup T(B)$$

\begin{figure}[h]
  \centering
  \includegraphics[width=.5\linewidth]{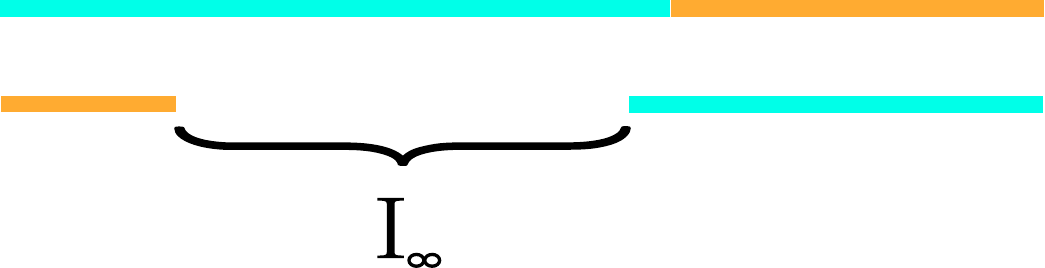}
\end{figure}

\noindent By definition, this interval is disjoint from $T(D)$, and
therefore
$$\forall x \in D \text{ and } n \in \mathbb N^*, \ T^n(x) \notin I_\infty$$
Moreover, our definition of Rauzy-Veech induction implies that any
point outside of the subinterval on which we consider the first return
will end up in this subinterval in finite time. Thus
$$\forall x \in D \text{ and } n \in \mathbb N, \exists k \in \mathbb N \text{ such that } T^k(x) \in I_n$$
Since we consider here only Rauzy-Veech induction procedure containing both infinitely right and left steps, we have that the orbit of any point of $D$ accumulates on
$\partial I_\infty$.\\

\noindent Let $\Omega$ be the complement of all the images of
$I_\infty$ namely
$$\Omega := D - \bigcup_{n=0}^\infty  T^n I_\infty$$
The measure of $I_\infty$ is $1/2$, taking the image by $T$ divides
the measure of any interval by two. Moreover, any two iterated images of this set
are disjoint since $I_\infty$ is disjoint of the image set of the injective application $T$, hence the measure of $\Omega$ is
$1/2 \cdot (1 + 1/2 + 1/2^2 + \dots) = 1$.  As we remarked, the orbit
of any point of $D$ accumulates to $\partial I_\infty$ and thus to any
image of it, hence to any point of $\partial \Omega$. As
$\bigcup_{n=0}^\infty T^n I_\infty$ has full measure, $\Omega$ has
zero Lebesgue measure and thus has empty interior. To conclude,
$\Omega$ is the limit set of any orbit of $T$.\\

Now $\Omega$ is closed set with empty interior. Moreover, if we take a
point in $\Omega$, any of its neighborhood contains some image of the interval $I_\infty$ since $\Omega$ has zero measure, and thus its
boundary. Hence no point is isolated, and $\Omega$ is a Cantor
set. Which leads to the following proposition,

\begin{proposition}
  \label{dynamics_desc}
  In the space of directions $[1,4]$ there is a set $\mathcal{H}$ (which is the union of the set $H$ constructed in this section union $]1,2[$) whose
  complement $A \cup B$ is a Cantor set of zero measure which satisfies

  \begin{itemize}
  \item $\forall \theta \in \mathcal{H}$ the foliation
    $\mathcal F_\theta$ is attracted by an attracting leaf;
    \item $A$ is countable, and $\forall \theta \in A$ the foliation
    $\mathcal F_\theta$ is attracted by a saddle connexion;
  \item $B$ is not countable, and $\forall \theta \in B$ the foliation
    $\mathcal F_\theta$ concentrates on a stable Cantor set.
  \end{itemize}
\end{proposition}

\section{Topological type of the elements of the Veech group}

\subsection{Thurston's theorem on multi-twists.}
We recall in this subsection a theorem of Thurston allowing the understanding of the topological type of the elements of a subgroup of $\MCG$ generated by a couple of multi-twists. Let $\alpha$ and $\beta$ be two multi-curve on $\Sigma$. We say that

\begin{itemize}
\item $\alpha$ and $\beta$ are \textit{tight} if they intersect transversally and if their intersection number is minimal in their isotopy class;
\item $\alpha$ and $\beta$ fill up $\Sigma$ if $\Sigma \setminus (\alpha \cup \beta)$ is a union of cells.
\end{itemize}
Denote by $\alpha_1, \cdots, \alpha_k$ and $\beta_1, \cdots, \beta_l$ the components of $\alpha$ and $\beta$ respectively. We form the $k \times l$ matrix $N = (i(\alpha_i, \beta_j))_{1\leq i \leq k, \ 1 \leq j \leq l}$. One easily checks that $\alpha \cup \beta$ is connected if and only if a power of $N^tN$ is positive. Under this assumption, $N^tN$ has a unique positive eigenvector $V$ of eigenvalue $\mu > 0$. We also denote by $T_{\alpha}$ (resp. $T_{\beta}$) the Dehn twist along $\alpha$ (resp. along $\beta$).

\begin{theorem}[Theorem 7 of \cite{Thurston}]

Let $\alpha$ and $\beta$ two multi-curves which are tight and which fill up $\Sigma$, and assume that $\alpha \cup \beta$ is connected. Denote by $G(\alpha, \beta)$ the subgroup of $\MCG$ generated by $T_{\alpha}$ and $T_{\beta}$. There is a representation $\rho : G(\alpha, \beta) \longrightarrow \mathrm{PSL}(2,\R)$ defined by

$$ \rho(T_{\alpha} ) = \begin{pmatrix}
1 & \mu^{\frac{1}{2}} \\
0 & 1
\end{pmatrix} \ \text{and} \ \rho(T_{\beta} ) = \begin{pmatrix}
1 & 0\\
-\mu^{\frac{1}{2}} & 1
\end{pmatrix} $$ such that $g$ is of finite order, reducible or pseudo-Anosov according to whether $\rho(g)$ is elliptic, parabolic or pseudo-Anosov.

\end{theorem}

\subsection{The case of $\disco$.}

We want to use Thurston's theorem to prove the

\begin{theorem}
  \label{classification}
  For all $f \in \mathrm{Affine}(\disco)$, $f$ is of finite order,
  reducible or pseudo-Anosov according to whether its image by the
  Fuchsian representation in $\mathrm{SL}(2,\R)$ is elliptic,
  parabolic or hyperbolic.
\end{theorem}

In Subsection \ref{subsecgroup}, we exhibited two elements of the Veech's groupe $\V$, \ $\begin{pmatrix}
  1 & 6\\
  0 & 1
\end{pmatrix}$ and $\begin{pmatrix}
  1 & 0\\
  \frac{3}{2}& 1
\end{pmatrix}, $
as the images by the Fuchsian representation $\rho_1$ corresponding to the Dehn twists along the curves $\alpha$ and
$\beta$ drawn in Figure \ref{cylinder1} :\\

\begin{figure}[h]
  \begin{subfigure}{.49\linewidth}
    \vspace{.5cm}
    \hspace{-.5cm}
    \includegraphics[scale=.6]{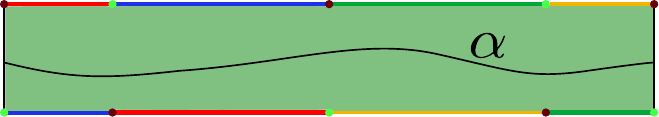}
    \subcaption{The curve $\alpha$.}
  \end{subfigure}
  \begin{subfigure}{.49\linewidth}
    \hspace{.5cm}
    \includegraphics[scale=.6]{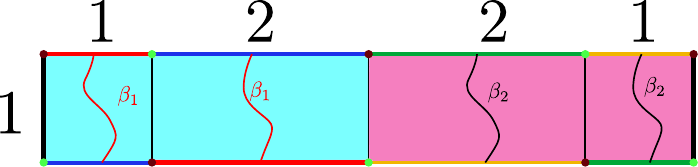}
    \caption{The multi-curve $\beta = \beta_1 \cup \beta_2$.}
  \end{subfigure}
  \caption{Definition of $\alpha$ and $\beta$}
  \label{cylinder1}
\end{figure}

\noindent One checks that:

\begin{itemize}

\item $\alpha \cup \beta$ is connected;

\item $\alpha$ and $\beta$ are tight since they can both be realized
  as geodesics of $\disco$;

\item $\alpha$ and $\beta$ are filling up $\disco$.

\end{itemize}

\noindent With an appropriate choice of orientation for $\alpha$ and
$\beta$, we have that $i(\alpha, \beta_1) = i(\alpha, \beta_2) =
2$. The intersection matrix associated is therefore
$N =\begin{pmatrix} 2 & 2
\end{pmatrix}$ and $N^tN = \begin{pmatrix} 8
\end{pmatrix}$. The parameter $\mu$ is then equal to $8$ and
$\sqrt{\mu}= 2\sqrt{2}$. We are left with two representations
$$ \rho_1, \rho_2 : G(\alpha, \beta) \longrightarrow \mathrm{PSL}(2, \R). $$

\begin{enumerate}

\item $\rho_1$ is the restriction of the Fuchsian representation to
  $G(\alpha, \beta) < \mathrm{Affine}(\disco)$ composed with the
  projection onto $\mathrm{PSL}(2, \R)$.

\item $\rho_2$ is the representation given by Thurston's theorem.

\end{enumerate}

By definition of these two representations, $\rho_1$ maps $T_{\alpha}$
to $\begin{pmatrix}
  1 & 6 \\
  0 & 1
\end{pmatrix}$ and $\rho_2$ maps it to $\begin{pmatrix}
  1 & 2\sqrt{2} \\
  0 & 1
\end{pmatrix}$; and $\rho_1$ maps $T_{\beta}$ to $\begin{pmatrix}
  1 & 0 \\
  \frac{3}{2} & 1
\end{pmatrix}$ and $\rho_2$ maps it to $\begin{pmatrix}
  1 & 0 \\
  -2\sqrt{2} & 1
\end{pmatrix}$.

\begin{proposition}
  \label{comparison}
  For all $g \in G(\alpha, \beta)$, $\rho_1(g)$ and $\rho_2(g)$ have
  same type.
\end{proposition}

\begin{proof} \

  \begin{itemize}

  \item $\rho_1$ and $\rho_2$ are faithful;

  \item $\rho_1$ and $\rho_2$ are Schottky subgroups of
    $\mathrm{PSL}(2,\R)$ of infinite covolume;

  \item $\rho_1$ and $\rho_2$ send $T_{\alpha}$ and $T_{\beta}$ to two
    parabolic elements;

  \end{itemize}

  \noindent As a consequence of these three facts, the quotient of
  $\mathbb{H}$ by the respective actions of $G(\alpha, \beta)$ through
  $\rho_1$ and $\rho_2$ respectively are both a sphere $S$ with two
  cusps and a funnel. No element of $\rho_1(G(\alpha, \beta))$ or
  $\rho_2(G(\alpha, \beta))$ is elliptic, and the image of
  $g \in G(\alpha, \beta)$ is parabolic in $\rho_1(G(\alpha, \beta))$
  or $\rho_2(G(\alpha, \beta))$ if and only if the corresponding
  element in $\pi_1(S)$ is in the free homotopy class of a simple
  closed curve circling a cusp. Which proves the proposition.
\end{proof}

There is little needed to complete the topological description of the
elements of the Veech group of $\disco$. Indeed, Proposition
\ref{comparison} above together with Thurston's theorem ensures that
the topological type of
$g \in G(\alpha, \beta) \subset \mathrm{Affine}(\disco)$ is determined
by (the projection to $\mathrm{PSL}(2,\R)$ of) its image by the
Fuchsian representation (namely $g$ has finite order if $\rho_1(g)$ is
elliptic, $g$ is reducible if $\rho_1(g)$ is parabolic and
pseudo-Anosov if $\rho_1(g)$ is hyperbolic).

\noindent The group $G(\alpha, \beta)$ has index $2$ in
$\discoveech$. The involution $i \in \mathrm{Affine}(\disco)$
acting as $\begin{pmatrix}
  -1 & 0 \\
  0 & -1
\end{pmatrix}$ preserves the multi-curves $\alpha$ and $\beta$ and
therefore commutes to the whole $G(\alpha, \beta)$. In particular, any
element of $\discoveech$ writes $g \cdot i$ with
$g \in G(\alpha, \beta)$. The type of $g \cdot i$ is the same as
the type of $g$ and this completes the classification.

\section{The global picture.}
\label{secglopic}

\noindent Gathering all materials developed in the previous sections, we prove
here the main theorems announced in the introduction.

\begin{proposition}
	\label{prop dyn tri}
  Assume that the foliation $\mathcal{F}_{\theta}$ of $\disco$ has a
  closed attracting leaf $F^+$. Then it has a unique repulsing leaf
  $F^-$ and any leaf which is different from $F^-$ and regular
  accumulates on $F^+$.
\end{proposition}

\noindent This proposition ensures that in all the cases where we have already
found an attracting leaf, the dynamics of the foliation is as simple
as can be.

\begin{proof} The proof consists in exhibiting the symmetry given by the element $- \mathrm{Id}$ of the Veech group in this case. If the foliation given by the direction $\theta$ has an attractive hyperbolic leaf then, from the dynamical study performed in previous sections, one can suppose, up to applying an element of the Veech group, that such a direction lies in $\mathcal{H}$. Therefore, any point $p \notin F^+$ in the stable subsurface (see Figure \ref{stable}) has to satisfy both the following:

\begin{itemize}
\item in positive time, the trajectory remains in the stable subsurface and is attracted to the closed leaf $F^+$;
\item in negative time, it escapes the stable subsurface at some point.
\end{itemize}
The first item is basically what we proved in Subsection \ref{subsec attractive}. The second one is proved by arguing on the $\omega$-limit of a point $p$, but for the foliation endowed with the reverse orientation. This set being invariant in both positive and negative time it has to be the closed leaf $F^+$ since it is the unique invariant set in the stable surface, see Subsection \ref{subsec attractive}. This implies $p \in F^+$, which concludes. \\

Therefore, in negative time, the point must visit the other stable subsurface corresponding to the colors red and yellow of Figure \ref{stable}, since the disco surface is the disjoint union of both these subsurfaces. This subsurface is stable for negative times, the point is therefore trapped in it. One can finally check that the dynamical study performed in Section \ref{subsec attractive} applies in the same way in this context, exchanging only the colors blue and green for red and yellow. In negative time, the point $p$ is then going to be attracted to the yellow red analogous leaf $F^-$, which is repulsive if seen as a leaf of the oriented foliation.
\end{proof}

We say that a direction having such a dynamical behavior is
\textit{dynamically trivial}.

\begin{corollary}
  The directions in $\Lambda_{\veechsg}$ are not dynamically trivial.
\end{corollary}

\begin{proof}
  The set of directions fixed by a hyperbolic element of
  $\Gamma$ is dense in $\Lambda_{\Gamma}$. Such a
  direction cannot be dynamically trivial, for otherwise the associated
  collection of closed leaves would be globally fixed by the corresponding
  element of $\mathrm{Affine}(\disco)$, which can not occur since, according to Theorem
  \ref{classification}, such an element is a pseudo-Anosov diffeomorphism. On the other hand, Proposition  \ref{prop dyn tri} shows that the set of dynamically trivial directions is the same as the set of directions admitting an attractive leaf, in particular both sets are open since the last one is. We conclude using the density in $\Lambda_{\Gamma}$ of the set of directions being not dynamically trivial: fixed points of hyperbolic matrices of $\Gamma$.
\end{proof}

\begin{theorem}
  The set of dynamically trivial directions in $S^1$ is open and has
  full measure.
\end{theorem}

\begin{proof}
  Recall that the definition of $\veechsg$ and $\V$ are given
  in Sections \ref{subsecgroup} and \ref{subsecgroup2}.
  Since $-\mathrm{Id}$ belongs to the Veech group of $\disco$, the foliations
  $\mathcal{F}_{\theta}$ and $\mathcal{F}_{-\theta}$ have the same
  dynamical behaviour. We will therefore consider parameters $\theta$
  in $\mathbb{RP}^1$ instead of in $S^1$. We denote then by
  $\mathcal{T} \subset \mathbb{RP}^1$ the set of dynamically trivial
  directions in $\mathcal{T}$. We have proved in Section
  \ref{secgendir} that the intersection of $\mathcal{T}$ and
  $J = \{ [1: t] \ | \ t \in [1,4] \} \subset \mathbb{RP}^1$ is the
  complement of a Cantor set and that $\mathcal{T} \cap J$ has full
  measure.

 \noindent  Also $J$ is a fundamental domain (see Proposition \ref{fundomain}) for the action of $\veechsg$ on
  $\Omega_{\veechsg}$ the discontinuity set of $\veechsg$.  Since
  $\veechsg < \V$, two directions in $\mathbb{RP}^1$ in the same orbit
  for the action of $\theta$ induce conjugated foliations on $\Sigma$
  and therefore have same dynamical behaviour. This implies that
  $\mathcal{T} \cap \Omega_{\veechsg}$ is open and has full measure in
  $\Omega_{\veechsg}$. Since $\Omega_{\veechsg}$ has itself full
  measure in $\mathbb{RP}^1$, $\mathcal{T}$ has full measure in
  $\mathbb{RP}^1$. The fact that it is open is a consequence of
  the stability of dynamically trivial foliations for the $\mathcal{C}^{\infty}$
  topology, see \cite{Liousse} for instance.
\end{proof}

\noindent Relying on a similar argument exploiting in a
straightforward manner the action of the Veech group and the depiction
of the dynamics made in Section \ref{secgendir}, we get

\begin{theorem}
There exists a Cantor set $\mathcal{K} \subset S^1$ such that for
all $\theta \in \mathcal{K} $, the foliation $\mathcal{F}_{\theta}$
accumulates to a set which is locally the product of a Cantor set
with an interval. Such a set always have zero Lebesgue measure.
\end{theorem}

We believe it is worth pointing out that the method we used to find
these \textit{'Cantor like'} directions is essentially different
compared to the one used in \cite{CamelierGutierrez}, \cite{BressaudHubertMaass}
and \cite{MarmiMoussaYoccoz}. Indeed they are proper attracting set in the
sense that they have an open neighbourhood in $\Sigma$ which is pushed
by the flow strictly within itself after a certain time.

\noindent These results allows us to give a complete description of $\veechsg$.

\begin{theorem}
  \label{veechgroup1}
  The Veech group of $\disco$ is exactly $\veechsg$.
\end{theorem}

\begin{proof}

  We divide the proof into four steps:

  \begin{enumerate}

  \item proving that any element in $\V$ preserves $\Lambda_{\Gamma}$;
    \label{1}

  \item proving that $\Gamma$ has finite index in $\V$;

  \item proving that the group $\Gamma$ is normal in $\V$;

  \item concluding.
  \end{enumerate}

  \vspace{2mm}

  (1) Let us prove the first point. Because of the description of the
  dynamics of the directional foliations we have achieved, one can
  show the limit set of the Veech group is the same as the limit set
  of $\veechsg$. If not, there must be a point of
  $\Lambda_{\textbf{V}_{\disco}}$ in the fundamental interval $I$. But
  since the group $\textbf{V}_{\disco}$ is non elementary it implies
  that we have to find in $I$ infinitely many copies of a fundamental
  domain for the action of $\Gamma$ on the discontinuity set. In
  particular infinitely many disjoint intervals corresponding to
  directions where the foliation has an attracting leaf of dilatation
  parameter $2$. But by the study performed in the above section the
  only sub-interval of $I$ having this property is $]1,2[$.

  \vspace{2mm}

  (2) The second point follows from the fact that the projection
  $$ \Gamma \backslash \Hy \longrightarrow \V \backslash \Hy $$
  induces an isometric orbifold covering

  $$ C(\Gamma \backslash \Hy) \longrightarrow  C(\V \backslash \Hy ).$$
  Since $C(\Gamma \backslash \Hy)$ has finite volume
  (see Section \ref{subsecfucgam}) and because
  $$[\Gamma: \V] = \frac{\mathrm{vol}(C(\Gamma \backslash
    \Hy))}{\mathrm{vol}(C(\V \backslash \Hy))}$$, this ratio must be
  finite and hence $\Gamma$ has finite index in $\V$.

  \vspace{2mm} (3) Remark that $\Gamma$ is generated by two parabolic
  elements $A$ and $B$ and that these define the only two conjugacy
  class in $\Gamma$ of parabolic elements. We are going to prove that
  any element $g \in \V$ normalise both $A$ and $B$. Since $\Gamma$
  has finite index in $\V$, there exists $n \geq 1$ such that
  $(gAg^{-1})^n \in \Gamma$. There are but two classes of conjugacy of
  parabolic elements in $\Gamma$ which are the ones of $A$ and $B$. If
  $n\geq 2$, this implies that $\V$ contains a strict divisor of $A$,
  which would make the limit set of $\V$ larger that
  $\Lambda_{\Gamma}$ (consider the eigenvalues of the matrix $AB^{-1}$
  which determine points in the boundary on the limit set, see Lemma
  \ref{pro bord cantor}). Therefore $gAg^{-1}$ belongs to $\Gamma$. A
  similar argument shows that $gBg^{-1} \in \Gamma$ and since $A$ and
  $B$ generate $\Gamma$, $g$ normalises $\Gamma$. Hence $\Gamma$ is
  normal in $\V$.  \vspace{2mm}

  (4) Any $g \in \V$ thus acts on the convex core
  $C(\Gamma \backslash \Hy)$ of the surface
  $C(\Gamma \backslash \Hy)$. In particular it has to preserve the
  boundary of $C(\Gamma \backslash \Hy)$, which is a single geodesic
  by Proposition \ref{proposition discontinuity set}. At the universal
  cover it means that $g$ has to fix a lift of the geodesic $c$, thus
  the isometry $g$ permutes two fixed points of a hyperbolic element
  $h$ of $\veechsg$. Two situations can occur:

  \begin{itemize}
  \item $g$ is an elliptic element. His action on
    $\Gamma \backslash \Hy$ cannot permute the two cusps because they
    correspond to two essentially different cylinder decompositions on
    $\Sigma$. It therefore fixes the two cusps and hence must be
    trivial.

  \item $g$ is hyperbolic and fixes the two fixed points of
    $h$. Moreover, by Lemma \ref{pro bord cantor}, it acts on the
    fundamental interval $I$ and as we discussed above such an action
    has to be trivial because of our study of the associated
    directional foliations, the translation length of $g$ is then the
    same than $h$. But $g$ is fully determined by its fixed points and
    its translation length, which shows that $g=h$ and thus
    $g \in \veechsg$.
  \end{itemize}

  \noindent Any element of $\V$ therefore belongs to $\Gamma$ and the
  theorem is proven.
\end{proof}

\begin{remark}
  This theorem implies that the completely periodic directions correspond to
  the orbit by the Veech group of the horizontal and vertical directions, since any parabolic element is
  conjugated to the Dehn twist in one of these two directions. This is the set we
  denoted by $\mathcal H$ in Theorem \ref{thm3}.
\end{remark}

\bibliographystyle{alpha}
\bibliography{biblio}
\end{document}